\documentclass[11pt,twoside]{article}
\usepackage{relsize,amsmath,amsfonts,amssymb,amsthm,graphicx,epsf,epstopdf,
color,pifont,subfigure,flafter,bm,amscd,multirow,setspace,tikz,caption,mathrsfs,enumerate,enumitem}
\allowdisplaybreaks[4]
\usepackage[numbers, sort]{natbib}

\newtheorem{lm}{Lemma}[section]
\newtheorem{thm}{Theorem}[section]

\newtheorem{prop}{{\bf Proposition}}[section]

\newcounter{saveeqn}%

\makeatletter
\evensidemargin
\oddsidemargin
\makeatother

\title{\Large\bf Bifurcations of grazing loops of arbitrary tangent multiplicity
in piecewise-smooth systems
}

\author{Xingwu Chen$^1$, ~~Zhihao Fang$^2$\!\!
\footnote{Corresponding author: Zhihao Fang (Fangmath010@outlook.com)}
, ~~Tao Li$^3$
\\
{\small 1. School of Mathematics, Sichuan University, Chengdu, Sichuan 610064, P. R. China}\\
{\small 2. School of Mathematics, China University of Mining and Technology,}\\
{\small Xuzhou, Jiangsu 221116, P. R. China}\\
{\small 3. School of Mathematics, Southwestern University of Finance and Economics,}\\
  {\small Chengdu, Sichuan 611130, P. R. China}
}
\date{}

\begin{document}
\maketitle


\begin{abstract}
In piecewise-smooth differential systems, a hyperbolic limit cycle of a
subsystem loses its structural stability if it grazes the switching manifold at
a tangent point. Such a cycle is called a grazing loop and in this paper
we investigate its bifurcations for arbitrary tangent multiplicity.
For the low-multiplicity tangency, the recurrences are comprehensively
captured by a functional perturbation with two parameters in previous
publications, where the parameters characterize the recurrences near the
tangent point and the limit cycle respectively. However, for high-multiplicity
tangency, these parameters fail to capture the recurrences and thus,
Poincar\'e return maps can not be defined as usual. To address these
challenges, we construct a functional perturbation with functions to
clarify the recurrences and simultaneously, propose a localization method
to make these two recurrences equivalent. We finally establish a quantitative
relationship between the multiplicity of tangency and the numbers of
crossing limit cycles, sliding loops bifurcating from the grazing loop and the number of tangent points on these sliding loops.

\vskip 0.2cm



\end{abstract}

\baselineskip 15pt
\parskip 10pt
\thispagestyle{empty}
\setcounter{page}{1}

\section{Introduction}

\setcounter{equation}{0}
\setcounter{lm}{0}
\setcounter{thm}{0}
\setcounter{rmk}{0}
\setcounter{df}{0}
\setcounter{cor}{0}
\setcounter{prop}{0}

Consider the following piecewise-smooth differential system (PWS system for
short)
\begin{equation}
    \begin{aligned}
            \left( \begin{array}{c}
            \dot{x}\\
            \dot{y}\\
        \end{array} \right) =\left\{
        \begin{aligned}
    &		\left( \begin{array}{c}
                f^+(x,y)\\
                g^+(x,y)
            \end{array} \right)   &&\qquad\mathrm{if}~(x,y)\in\Sigma^+,\\
    &		\left( \begin{array}{c}
                f^-(x,y)\\
                g^-(x,y)
            \end{array} \right)   &&\qquad \mathrm{if}~(x,y)\in\Sigma^-,
        \end{aligned} \right.
        \label{pws1}
    \end{aligned}
\end{equation}
where $f^\pm(x,y)$, $g^\pm(x,y)\in C^{\infty}(\cal U)$ and
$\Sigma^\pm:=\left\{(x,y)\in{\cal U}:~y\gtrless0\right\}$ for a bounded
open set ${\cal U}\subset{\mathbb R}^2$ containing the origin $O$.
As in \cite{Teixeira11}, the set $\Sigma:=\left\{(x,y)\in{\cal U}:~y=0\right\}$
is called a {\it switching manifold}. Throughout this paper, the smooth
system defined on $\Sigma^+$ (resp. $\Sigma^-$) is called the
{\it upper subsystem} (resp. {\it lower subsystem}).
As indicated in \cite{Filippov88}, the solution of system~\eqref{pws1} is an
absolutely continuous function $(x(t),y(t))^\top$ defined over some
interval $I$ which satisfies the following differential inclusion
\begin{equation*}
\left(
\begin{aligned}
&\dot x \\
&\dot y
\end{aligned}
\right)\in
F(x,y):=\left\{
\begin{aligned}
&\left\{{\cal Z}^+(x,y)\right\},  &&(x,y)\in\Sigma^+, \\
&\left\{a{\cal Z}^+(x,y)+(1-a){\cal Z}^-(x,y):~a\in[0,1]\right\}, &&(x,y)\in\Sigma, \\
&\left\{{\cal Z}^-(x,y)\right\},  &&(x,y)\in\Sigma^-
\end{aligned}
\right.
\end{equation*}
almost everywhere over $I$, where ${\cal Z}^\pm(x,y):=\left(f^\pm(x,y),g^\pm(x,y)\right)$
denote the upper and lower vector field of system~\eqref{pws1} respectively.
Further, the orbit of system~\eqref{pws1} is the set
$\{(x(t),y(t)):~t\in I\}\subset{\mathbb R}^2$.

As indicated in \cite{Kuznetsov03}, function $h(x):=g^+(x,0)g^-(x,0)$ is defined to
characterize dynamical behaviors near $\Sigma$. It is not hard to check
that $h(x)>0$ implies the orbits of one subsystem reach $\Sigma$ and the orbits of
the other subsystem escape from $\Sigma$. Thus, orbits of system~\eqref{pws1} cross
$\Sigma$, i.e., the {\it crossing motions}, and further, the set
$\Sigma_c:=\left\{(x,y)\in\Sigma:~h(x)>0\right\}$ is called the {\it crossing region}
as indicated in \cite{Kuznetsov03,Teixeira11,Filippov88}.
On the other hand, $h(x)<0$ implies that the orbits of two subsystems reach or
escape from $\Sigma$ simultaneously. Further, the set $\Sigma_s:=\left\{(x,y)\in\Sigma:~h(x)<0\right\}$
is called the {\it sliding region} and a {\it sliding vector field}
is defined on $\Sigma_s$ as in \cite{Kuznetsov03,Teixeira11,Filippov88}.
The orbits of two subsystems reach $\Sigma_s$ and connect the sliding orbits,
i.e., the {\it sliding motions}.

A point $p:(x_0,0)\in\Sigma$ is called a {\it tangent point} if $h(x_0)=0$
and it is not an equilibrium of subsystems (see, e.g., \cite{Filippov88,Kuznetsov03,Teixeira11}).
By the visibility of the tangent orbits, tangent points can be
divided into $4$ distinct types as shown in Figure~\ref{Fig-TPV},
i.e., visible one, invisible one, left one and right one, as in \cite{CF25}.
\begin{figure}[h]
\centering
\subfigure[visible]
 {  \scalebox{0.34}[0.34]{
   \includegraphics{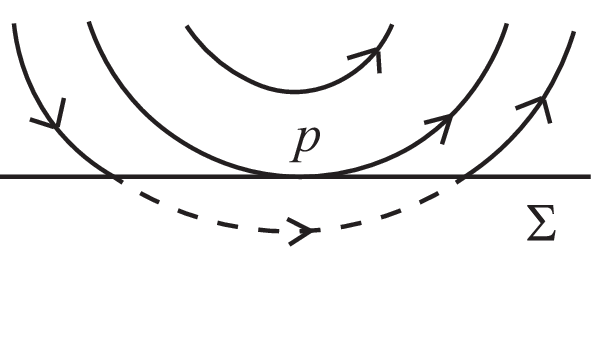}}}
\subfigure[invisible]
 {  \scalebox{0.34}[0.34]{
   \includegraphics{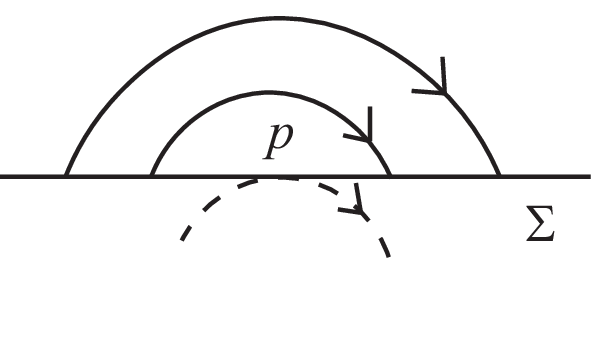}}}
\subfigure[left]
 {  \scalebox{0.34}[0.34]{
   \includegraphics{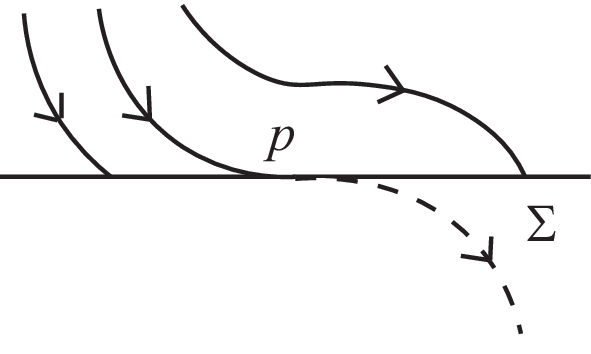}}}
\subfigure[right]
 {  \scalebox{0.34}[0.34]{
   \includegraphics{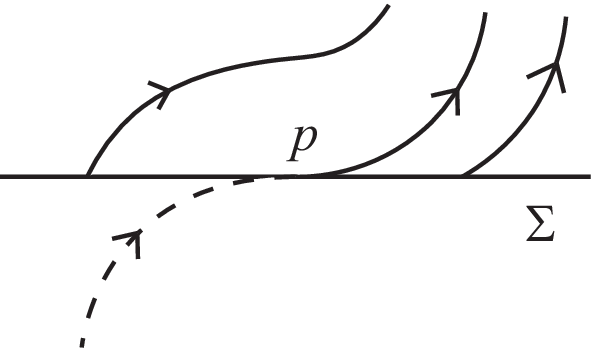}}}
   \caption{The visibility of tangent points}
\label{Fig-TPV}
\end{figure}
Clearly, tangent point $p:(x_0,0)$ corresponds a zero point $x_0$ of $g^+(x,0)$ or
$g^-(x,0)$. Thus, as indicated in \cite{CF25} a tangent point $p:(x_0,0)$
of system~\eqref{pws1} is called to be of multiplicity $(m^+,m^-)$ if
$x_0$ is a zero of $g^+(x,0)$ (resp. $g^-(x,0)$) of multiplicity $m^+$ (resp. $m^-$), i.e.,
\begin{eqnarray*}
&&g^+(x_0,0)=\frac{\partial g^+}{\partial x}(x_0,0)=...=\frac{\partial^{(m^+-1)}g^+}{\partial x^{(m^+-1)}}(x_0,0)=0,~\frac{\partial^{m^+}g^+}{\partial x^{m^+}}(x_0,0)\ne 0,\\
&&g^-(x_0,0)=\frac{\partial g^-}{\partial x}(x_0,0)=...=\frac{\partial^{(m^--1)}g^-}{\partial x^{(m^--1)}}(x_0,0)=0,~\frac{\partial^{m^-}g^-}{\partial x^{m^-}}(x_0,0)\ne 0,
\end{eqnarray*}
where non-negative integers $m^\pm$ satisfy $m^++m^-\ge 1$.
It is not hard to check that a tangent point is visible or invisible
(resp. left or right) if and only if $m^+$ is odd (resp. even).
Analysis is similar for the lower subsystem and we omit the statements.
A tangent point of multiplicity $(1,0)$ (resp. $(0,1)$) is generally
called a {\it fold} of the upper (resp. lower) subsystem (see, e.g., \cite{Kuznetsov03,Teixeira11}).
A tangent point of multiplicity $(2,0)$
(resp. $(0,2)$) is generally called a {\it cusp} of the upper (resp. lower)
subsystem (see, e.g., \cite{Kuznetsov03,Teixeira11}).
Under perturbation, a tangent point can break into several tangent points,
some new local recurrence may occur with respect to $\Sigma$ and
some new invariant sets such as crossing limit cycles may appear.
Here ``crossing limit cycle'' means an isolated oriented
Jordan curve composed of two regular orbits of subsystems and two
crossing points.
Bifurcations of tangent points can be found in \cite{Bonet18,Kuznetsov03,Teixeira11,Ponce22,Han12,Zhang10,Siller} for
multiplicity $(1,1)$ and $(2,0)$, in \cite{CF21,Teixeira11} for multiplicity $(2,1)$ and in
\cite{Esteban23,Novaes25,CF25,Novaes21,Buzzi23} for general multiplicity $(m^+,m^-)$.

It is well-known that a hyperbolic limit cycle is structurally stable
in smooth differential systems, i.e., it is preserved under small perturbations
(see, e.g., \cite{KuznetsovBook}). However, in piecewise-smooth
system (\ref{pws1}), it is shown in \cite{Bernardo01,Bernardo98,ChenH18,Kuznetsov03,Teixeira11,Li20,Han13,HuangGW}
that a hyperbolic limit cycle of a smooth subsystem is no longer structurally
stable if it grazes $\Sigma$ at a tangent point as shown in
Figure~\ref{Fig-Gra} and, it is usually called a {\it grazing loop}
(or to be {\it grazing}) (see, e.g., \cite{Kuznetsov03,Teixeira11}).
\begin{figure}[h]
\centering
\subfigure[$g^-(0,0)>0$]
 {
  \scalebox{0.34}[0.34]{
   \includegraphics{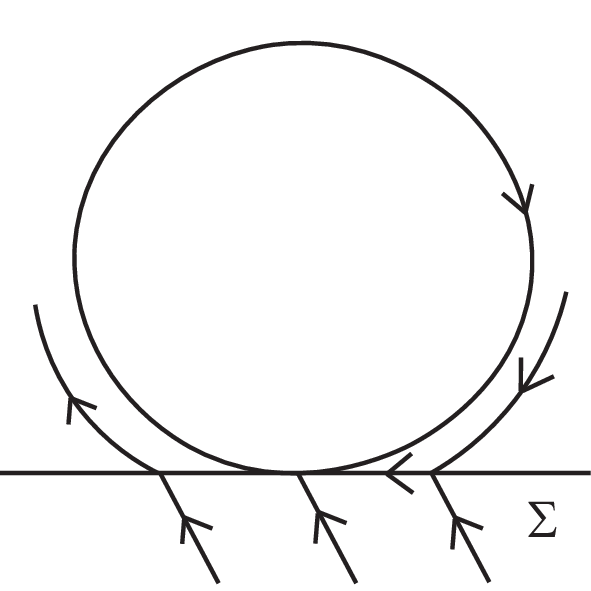}}}
\subfigure[$g^-(0,0)<0$]
 {
  \scalebox{0.34}[0.34]{
   \includegraphics{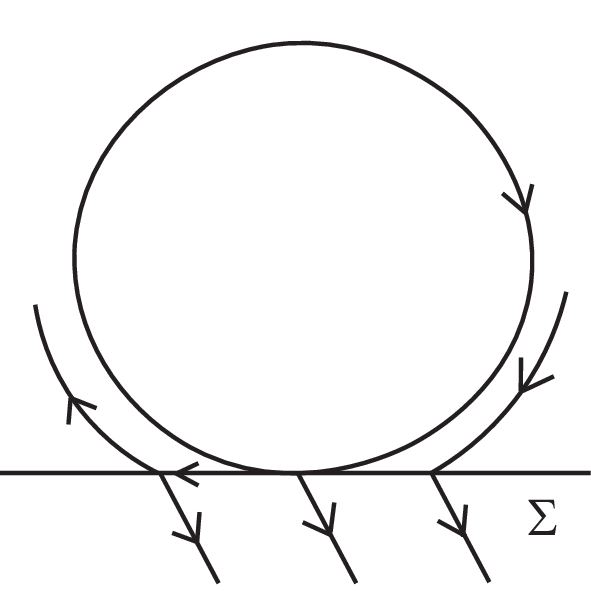}}}
\subfigure[$f^-(0,0)g^-_x(0,0)>0$]
 {
  \scalebox{0.34}[0.34]{
   \includegraphics{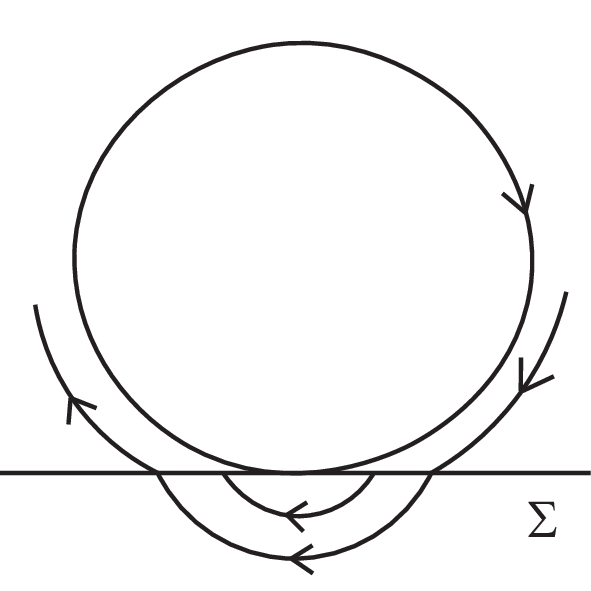}}}
\subfigure[$f^-(0,0)g^-_x(0,0)<0$]
 {
  \scalebox{0.34}[0.34]{
   \includegraphics{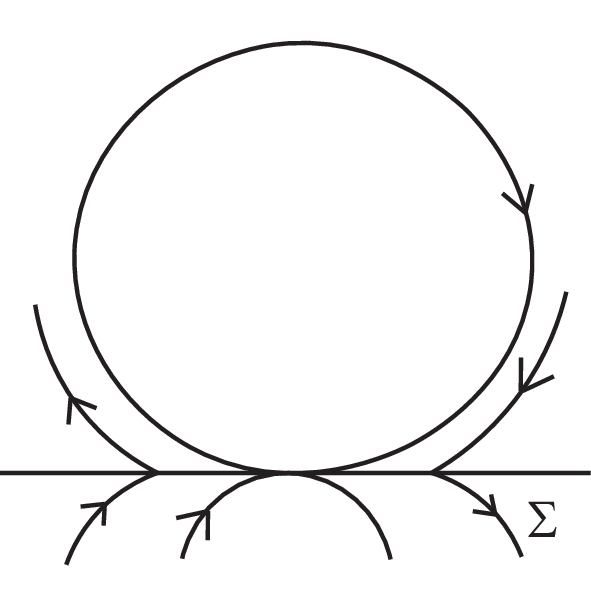}}}
   \caption{The examples illustrating distinguish grazing loops}
\label{Fig-Gra}
\end{figure}

In previous publications, people focus on investigating dynamical behaviors
of the grazing loops shown in Figure~\ref{Fig-Gra}. They are clockwise hyperbolic
limit cycles of the upper subsystem and tangent with $\Sigma$ at the origin $O$
which is a tangent point of multiplicity $(1,0)$ or $(1,1)$.
The grazing loops in Figure~\ref{Fig-Gra}(a)(b) can undergo two
distinct {\it grazing-sliding} bifurcations for different stability.
Precisely, there is either one sliding loop connecting one tangent point or
one standard limit cycle under some perturbations if it is a stable limit
cycle for (a). Here ``standard limit cycle'' means the limit cycle in $\Sigma^+$
or $\Sigma^-$ and ``sliding loop'' means an oriented Jordan curve composed of
regular orbits of subsystems and sliding orbits. In contrast, if it is an
unstable limit cycle, a sliding loop connecting one tangent point and one
standard limit cycle appear simultaneously under some perturbations
(see e.g., \cite{Kuznetsov03,Teixeira11,Kunze,Bernardo08,Bernardo082}).
Dynamical behaviors are similar for (b) and we omit the statements.
The grazing loops in Figure~\ref{Fig-Gra}(c)(d)  can undergo
two distinct {\it degenerate grazing-sliding} bifurcations.
It is proved in \cite{Han13} that there exist some perturbations such that
there is a crossing limit cycle and one standard limit cycle bifurcating from
the grazing loop in Figure~\ref{Fig-Gra}(c). Later, a complete bifurcation
diagram is obtained in \cite{Li20} and it is shown that there exist some
perturbations such that there are two crossing limit cycles
and one standard limit cycle. Meanwhile, there also exist some perturbations
such that there is one crossing limit cycle, one sliding loop connecting one
tangent point and one standard limit cycle.
For the grazing loop in Figure~\ref{Fig-Gra}(d), a complete
bifurcation diagram is given in \cite{Li20} and it is shown
that there exist some perturbations such that there is one
sliding loop connecting one tangent point and one standard limit cycle.
Recently, the grazing loop connecting a tangent point of multiplicity 
$(1,m)$ for general $m\ge2$ is investigated in \cite{CFSCM}. It is proved 
that there exist some perturbations such that there are at least 
$m$ crossing limit cycles and sliding loops bifurcating from the grazing loop,
i.e., a relationship between the multiplicity $m$ and the number of 
crossing limit cycles and sliding loops is established. More importantly,
\cite{CFSCM} uncovers the mechanism by which crossing limit cycles 
and sliding loops arise from the splitting of the tangent points, 
and proposes sufficient conditions for this mechanism to occur.

Note that bifurcations of grazing loops in previous publications focus
on the cases where the limit cycle is quadratically tangent to $\Sigma$. 
Thus, it is natural to consider the following question.
\begin{enumerate}[topsep=-3pt,itemsep=-3pt,partopsep=-3pt]
    \item[(Q)]
{\it How about bifurcations of grazing loops of arbitrary-multiplicity tangency
with $\Sigma$?  }
\end{enumerate}
In this paper, we focus on this question and investigate bifurcations
associated with two distinct types of grazing loops. One type corresponds
to the configuration in Figure~\ref{Fig-Gra}(a)(b), where the orbit passing
through $O$ of lower subsystem is transversal with $\Sigma$.
The other type corresponds to the configuration in Figure~\ref{Fig-Gra}(c)(d),
where the orbit passing through $O$ of the lower subsystem is tangent with
$\Sigma$ quadratically. We aim to establish a quantitative relationship
between the multiplicity of tangency and the numbers of crossing limit
cycles and sliding loops bifurcating from the grazing loop. 
Meanwhile, we also aim to investigate whether these grazing loops exhibit 
the mechanism given in \cite{CFSCM}.

This paper is organized as follows. The main results and difficulties
are stated in section 2. The perturbation functions are analyzed
in section 3. Some preliminary lemmas and proof of main results are
stated in section 4 and section 5 respectively. Finally,
we summarize conclusions and give some discussion remarks in section 6
to end this paper.

\section{Technique challenges and main results}

\setcounter{equation}{0}
\setcounter{lm}{0}
\setcounter{thm}{0}
\setcounter{rmk}{0}
\setcounter{df}{0}
\setcounter{cor}{0}

In this paper, we consider the bifurcations of the grazing loop connecting
one tangent point of multiplicity $(2k+1,0)$ and $(2k+1,1)$ for general
$k\in{\mathbb Z}^+$. There are two main difficulties for question (Q)
arising from the arbitrary multiplicity of the tangent point.
One difficulty is that the change of local recurrence is very complicated
under perturbations. Although this difficulty is addressed by a functional
perturbation with functions in \cite{CF25}, the characterization of
the recurrence is not precise enough to analyze the Poincar\'e return map,
particularly the locations of tangent points. We propose a new functional
perturbation with functions to overcome this difficulty, which is shown in
sections 3 and 4. The other difficulty is that the relative positions between
the limit cycle and $\Sigma$ are very complicated under perturbations.
As discussion in section 1, there are several tangent points appearing and
consequently, there are several transversal or tangent intersections between
the limit cycle and $\Sigma$. Each type of these intersections corresponds
to a type of global recurrences and there is no method in the previous
publications to investigate this. We propose a localization method to
overcome this difficulty by establishing an equivalent relationship between
the local recurrences and the global recurrences. This method is shown in
section 5 and is called {\it localization of perturbations} in what follows.

Then we state the main results of this paper and come up with the
following assumptions.
\begin{enumerate}[topsep=-3pt,itemsep=-3pt,partopsep=-3pt]
    \item[{\bf A1}] The upper subsystem has a clockwise hyperbolic limit cycle
    which grazes $\Sigma$ at the origin $O$ and $O$ is a tangent point of
    multiplicity $m^+=2k+1$ for $k\in{\mathbb Z}^+$.
    \item[{\bf A2}] The origin $O$ is a transversal point for the lower
    subsystem, i.e., $g^-(0,0)\ne0$ or equivalently $O$ is a $0$-multiplicity tangent point for the lower subsystem.
    \item[{\bf A3}] The origin $O$ is a $1$-multiplicity tangent point for
    the lower subsystem and there are only crossing regions near $O$.
\end{enumerate}
It is not hard to check that system~\eqref{pws1} has a grazing loop with
the configurations shown in Figure~\ref{Fig-Gra}(a)(b)(resp. (c)(d)),
if it satisfies assumptions {\bf A1} and {\bf A2} (resp. {\bf A1} and {\bf A3}).
In what follows, they are denoted as $L_I$ and $L_{II}$ respectively.
Note that the assumption ${\bf A1}$ requires that $L_I$ and
$L_{II}$ are hyperbolic limit cycles of the upper subsystem. As indicated in
\cite{KuznetsovBook}, the hyperbolicity is characterized by $\sigma\ne 0$, where
\begin{equation}
    \sigma:=\int_{0}^{T}\left(\frac{\partial f^+(\gamma^+(s,0,0))}{\partial x}+\frac{\partial g^+(\gamma^+(s,0,0))}{\partial y}\right)ds.
    \label{sigmadef}
\end{equation}
Here $\gamma^+(t,x_0,y_0)$ denotes
the orbit of the upper subsystem of \eqref{pws1} with initial value $(x_0,y_0)$
and $T$ is the period of this limit cycle. For the grazing loop connecting
a tangent point of multiplicity $(2k+1,1)$, beside the types shown in
Figure~\ref{Fig-Gra}(c)(d) there are some other types, e.g., $f^-(0,0)>0$ and $g^-_x(0,0)>0$.
We omit these types because we focus on finding crossing periodic orbits
and the existence of crossing regions is necessary. As indicated in
\cite{CF25}, a visible tangent point of multiplicity $2k+1$ can break
into $2k+1$ tangent point of multiplicity $1$. Moreover, the visibility
of these tangent points follows an alternating pattern: visible, invisible,
visible, and so on. Thus, there are at most $k+1$ visible tangent points
bifurcating from $O$. For both two types of grazing loops $L_I$ and
$L_{II}$, all analysis is based on the appearance of all potential
tangent points.

For grazing loop $L_I$ (i.e., assumptions {\bf A1} and {\bf A2} hold), we have the following result.
\begin{thm}
Assume that system~\eqref{pws1} has a grazing loop $L_I$ for some
integer $k\ge1$. Then for each $\ell\in\{1,...,k+1\}$ there exist perturbations
of system~\eqref{pws1} such that there is a sliding loop connecting exactly
$\ell$ tangent points, and one additional standard limit cycle
if $g^-(0,0)$ and $\sigma$ have the same sign additionally, 
where $\sigma$ is given in \eqref{sigmadef}.
\label{thm1}
\end{thm}

As mentioned in section 1, the results given in \cite{Kuznetsov03} correspond
to the cases that $k=0$, i.e., there is a sliding loop connecting one tangent
point if $\sigma<0$, and one additional standard limit cycle if $\sigma>0$.
Thus, Theorem~\ref{thm1} generalizes the number of tangent point
on the sliding loop from the case $k=0$ to the case $k\ge1$ and
contributes the relationship between the number and the multiplicity of $O$.
Since the $k+1$ is the maximal number of visible tangent points,
the result in Theorem~\ref{thm1} is sharp with respect to the number of
tangent points on the sliding loop.

For grazing loop $L_{II}$ (i.e., assumptions {\bf A1} and {\bf A3} hold), let $\beta_c$, $\beta_s$ denote the numbers of
crossing limit cycles, sliding loops bifurcating from it respectively, and
we have the following result.
\begin{thm}
    Assume that system~\eqref{pws1} has a grazing loop $L_{II}$ for some
    integer $k\ge1$. Then the following statements hold.
    \begin{enumerate}[topsep=-3pt,itemsep=-3pt,partopsep=-3pt]
    \item[{\rm (a)}]For grazing loop $L_{II}$ in {\rm Figure~\ref{Fig-Gra}(c)}
    and each $\ell\in\{1+\lfloor k/2\rfloor,...,2+2\lfloor k/2\rfloor\}$,
    there exist perturbations of system~\eqref{pws1} such that there is
    one standard limit cycle and
    \begin{equation}
    \beta_c\ge\ell,~~~~~~\beta_s=2+2\lfloor k/2\rfloor-\ell,
    \label{Res}
    \end{equation}
  where $\lfloor k/2\rfloor$ denotes the maximal integer no greater than $k/2$.
    \item[{\rm (b)}]For grazing loop $L_{II}$ in {\rm Figure~\ref{Fig-Gra}(d)}
    and each $\ell\in\{1,...,k+1\}$,
    there exist perturbations of system~\eqref{pws1} such that there is
    one standard limit cycle and a sliding loop connecting exactly $\ell$
    tangent points.
    \end{enumerate}
    \label{thm2}
\end{thm}

As mentioned in section 1, it is not hard to check that the results given in
\cite{Li20,Han13} correspond to the cases $k=0$ and agree with
\eqref{Res} for $k=0$. Thus, Theorem~\ref{thm2}(a) generalizes the number
of crossing limit cycles and sliding loops from the case $k=0$ to the case $k\ge1$ and
contributes the relationship between the number and the multiplicity of $O$.
We believe that there is still room for improvement in conclusion (a) of
Theorem~\ref{thm2}, specifically $\beta_c$ in \eqref{Res}.
Similar to Theorem~\ref{thm1}, the conclusion (b) of Theorem~\ref{thm2}
contributes the relationship between the number of tangent points on
the sliding loop bifurcating from $L_{II}$ of configuration {\rm Figure~\ref{Fig-Gra}(d)} and the multiplicity of
$O$. There is no crossing limit cycle bifurcating from $L_{II}$ of configuration {\rm Figure~\ref{Fig-Gra}(d)} because
there is no recurrence region under perturbations, which is the same situation as $L_I$.

\section{Analysis of the perturbation functions}
\setcounter{equation}{0}
\setcounter{lm}{0}
\setcounter{thm}{0}
\setcounter{rmk}{0}
\setcounter{df}{0}
\setcounter{cor}{0}

As indicated in section 2, the Theorem~\ref{thm1} and \ref{thm2} are
based on the existences of some specific perturbations of system~\eqref{pws1}.
In this section, constructions of the perturbations and some corresponding
properties are given.

Via the following two cut-off functions
\begin{equation*}
h^*(x,r_1,r_2):=\left\{
\begin{aligned}
&0,&&x\in(-\infty,r_1],\\
&\frac{e^{-\frac{1}{x-r_1}}}{e^{-\frac{1}{x-r_1}}+e^{\frac{1}{x-r_2}}},&&x\in(r_1,r_2),\\
&1,&&x\in[r_2,\infty),
\end{aligned}
\right.
\end{equation*}
and
\begin{equation*}
h_*(x,r_1,r_2):=\left\{
\begin{aligned}
&1,&&x\in(-\infty,r_1],\\
&\frac{e^{\frac{1}{x-r_2}}}{e^{-\frac{1}{x-r_1}}+e^{\frac{1}{x-r_2}}},&&x\in(r_1,r_2),\\
&0,&&x\in[r_2,\infty),
\end{aligned}
\right.
\end{equation*}
we define
\begin{equation}
    H(x,{\boldsymbol l}):=\left\{
    \begin{aligned}
    &0,&&x\in(-\infty,l_1]\cup(l_2,\infty),\\
    &l_3\int_{l_1}^{x}h^*(s,l_1,l_1+d),&&x\in(l_1,l_1+d],\\
    &l_3\left(\frac{d}{2}+\int_{l_1+d}^{x}h_*(s,l_1+d,l_1+2d)ds\right),&&x\in(l_1+d,l_1+2d],\\
    &l_3\left(d-\int_{l_1+2d}^{x}h^*(s,l_1+2d,l_1+3d)ds\right),&&x\in(l_1+2d,l_1+3d],\\
    &l_3\left(\frac{d}{2}-\int_{l_1+3d}^{x}h_*(s,l_1+3d,l_2)ds\right),&&x\in(l_1+3d,l_2],\\
    \end{aligned}
    \right.
\label{df-H}
\end{equation}
where the vector ${\boldsymbol l}=(l_i)\in{\mathbb R}^3$ satisfying
$l_1<l_2$ and $d=(l_2-l_1)/4$.

\begin{prop}
For $H(x,{\boldsymbol l})$ defined in~\eqref{df-H}, the following statements hold.
\begin{enumerate}
\item[{\rm (a)}] $H(x,{\boldsymbol l})$ is $C^{\infty}$ with respect to $x$ for
$x\in{\mathbb R}$.
\item[{\rm (b)}] $H(x,{\boldsymbol l})$ and $\dot H(x,{\boldsymbol l})$ uniformly
converge to the zero function for $x\in{\mathbb R}$ as ${\boldsymbol l}\to {\boldsymbol 0}$.
\item[{\rm (c)}] $H(x,{\boldsymbol l})$ is integrable over ${\mathbb R}$ and
$$|l_3|d^2<\left|\int_{-\infty}^{\infty}H(x,{\boldsymbol l})dx\right|<4|l_3|d^2.$$
\end{enumerate}
\label{prop-1}
\end{prop}
\begin{proof}
As indicated in \cite{Lee18}, $h^*(x,r_1,r_2)$ and $h_*(x,r_1,r_2)$ are
$C^{\infty}$ over ${\mathbb R}$. Thus, $H(x,{\boldsymbol l})$ is
$C^{\infty}$ for $x\in(-\infty,l_1)\cup(l_2,\infty)\bigcup_{i=1}^{4}(l_1+(i-1)d,l_1+id]$.
It is not hard to obtain that
\begin{equation*}
\begin{aligned}
H(l_1+d,{\boldsymbol l})
& = l_3\int_{l_1}^{l_1+d}h(x,l_1,l_1+d)dx\\
& = l_3\int_{l_1}^{l_1+d}\frac{e^{-\frac{1}{x-l_1}}+e^{\frac{1}{x-l_1-d}}-e^{\frac{1}{x-l_1-d}}}{e^{-\frac{1}{x-l_1}}+e^{\frac{1}{x-l_1-d}}}dx\\
& = l_3d-l_3\int_{l_1}^{l_1+d}\frac{e^{\frac{1}{x-l_1-d}}}{e^{-\frac{1}{x-l_1}}+e^{\frac{1}{x-l_1-d}}}dx\\
& = l_3d+l_3\int_{l_1+d}^{l_1}\frac{e^{-\frac{1}{y-l_1}}}{e^{\frac{1}{y-l_1-d}}+e^{-\frac{1}{y-l_1}}}dy\\
& = l_3d-l_3\int_{l_1}^{l_1+d}h(x,l_1,l_1+d)dx\\
& = \frac{l_3d}{2},
\end{aligned}
\end{equation*}
where $y=2l_1+d-x$. Since $H(x,{\boldsymbol l})\to l_3d/2$ as $x\to l_1+d$
for $x\in(l_1+d,l_1+2d)$, $H(x,{\boldsymbol l})$ is continuous at $x=l_1+d$.
Further, it is not difficult to check that
\begin{equation*}
\frac{d H(x,{\boldsymbol l})}{dx}\to l_3,~~~~~~\frac{d^i H(x,{\boldsymbol l})}{dx^i}\to 0~(i=2,...)
\end{equation*}
as $x\to l_1+d$, which implies that $H(x,{\boldsymbol l})$ is $C^{\infty}$ at
$x=l_1+d$. It can be proved by a similar way that $H(x,{\boldsymbol l})$
is $C^{\infty}$ at $x=l_1+id$ ($i=2,3,4$), which finishes conclusion (a).

The uniform convergence can be proved by
$|H(x,{\boldsymbol l})|\le |l_3|d$ and $|\dot H(x,{\boldsymbol l})|\le |l_3|$
for any $x\in{\mathbb R}$. The integrability of $H(x,{\boldsymbol l})$ can be proved by
the continuity. It is not hard to obtain that
$\int_{-\infty}^{\infty}H(x,{\boldsymbol l})dx=\int_{l_1}^{l_2}H(x,{\boldsymbol l})dx.$
For the cases that $l_3>0$, it can be proved by $H(x,{\boldsymbol l})\le l_3d$ that
$\int_{l_1}^{l_2}H(x,{\boldsymbol l})dx<(l_2-l_1)l_3d=4l_3d^2.$
Let $L_1$ (resp. $L_2$) denote the line segment connecting the point
$(l_1+2d,l_3d)$ and $(l_1+d,0)$ (resp. $(l_1+3d,0)$).
\begin{figure}[htp]
\centering
\includegraphics[scale=0.45]{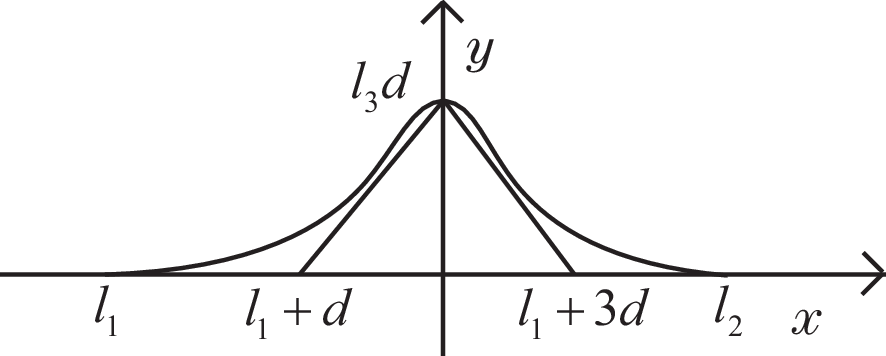}
\caption{The triangle lying below $H(x,{\boldsymbol l})$}
\label{Fig-Tri}
\end{figure}
There is a triangle
composed of $L_1$, $L_2$ and $\{(x,0):~x\in[l_1+d,l_1+3d]\}$ in the
interior of region surrounded by $\{(x,y):~x\in[l_1,l_2],y=H(x,{\boldsymbol l})\}$
and $\{(x,0):~x\in[l_1,l_2]\}$ as shown in Figure~\ref{Fig-Tri}.
This implies that
$l_3d^2<\int_{l_1}^{l_2}H(x,{\boldsymbol l})dx.$
The analysis of $l_3<0$ is similar and we omit the statements.
\end{proof}
Proposition \ref{prop-1}(b) implies that $H(x,{\boldsymbol l})$
can be defined as the zero function for ${\boldsymbol l}={\boldsymbol 0}$.
For some integer $s>0$, we define
\begin{equation}
\psi(x,{\boldsymbol l}):=\sum_{i=1}^{s}H_i(x,{\boldsymbol l}_i),
\label{df-psi}
\end{equation}
where ${\boldsymbol l}:=({\boldsymbol l}_1,...,{\boldsymbol l}_s)$ for
${\boldsymbol l}_i:=(l_{i,1},l_{i,2},l_{i,3})$ satisfying
$l_{i,1}<l_{i,2}$ and functions $H_i(x,{\boldsymbol l}_i)$ ($i=1,...,s$)
take the form \eqref{df-H}. It can be directly proved
by proposition~\ref{prop-1} that $\psi(x,{\boldsymbol l})$ is $C^{\infty}$ over
${\mathbb R}$ and that $\psi(x,{\boldsymbol l}), \dot\psi(x,{\boldsymbol l})$
uniformly converge to the zero function as ${\boldsymbol l}\to{\boldsymbol 0}$
because of the finite summation. An example of $\psi(x,{\boldsymbol l})$
are given in Figure~\ref{Fig-Psi}.
\begin{figure}[htp]
\centering
\includegraphics[scale=0.45]{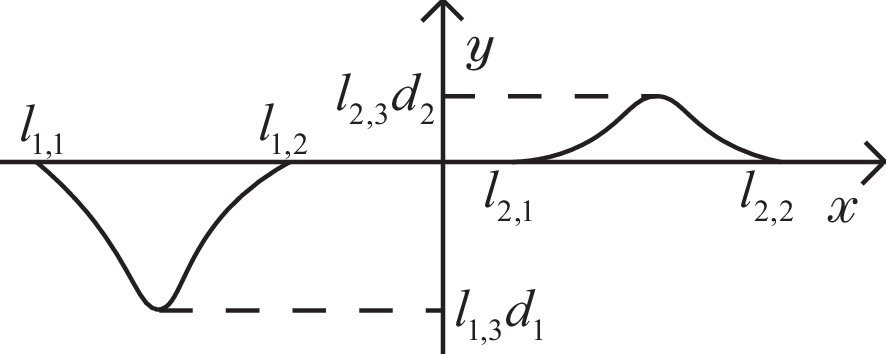}
\caption{$s=2$ and $l_{1,3}<0$, $l_{2,3}>0$}
\label{Fig-Psi}
\end{figure}

For PWS system~\eqref{pws1}, we define the following perturbation systems by
using $\psi(x,{\boldsymbol l})$
\begin{equation}
    \begin{aligned}
            \left(\begin{array}{c}
            \dot{x}\\
            \dot{y}\\
        \end{array}\right)=\left\{
        \begin{aligned}
    &       \left(\begin{array}{c}
                f^+(x,y)\\
                g^+(x,y)+\psi^+(x,{\boldsymbol l}^+)
            \end{array}\right)   &&\mathrm{if}~(x,y)\in\Sigma^+,\\
    &       \left(\begin{array}{c}
                f^-(x,y)\\
                g^-(x,y)+\psi^-(x,{\boldsymbol l}^-)
            \end{array}\right)   &&\mathrm{if}~(x,y)\in\Sigma^-,
        \end{aligned} \right.
    \end{aligned}
\label{pws1-exunfold}
\end{equation}
where $\psi^\pm(x,{\boldsymbol l}^\pm)$ take form of $\psi(x,{\boldsymbol l})$
in \eqref{df-psi}. In the following, some analysis of system~\eqref{pws1-exunfold}
is given to end this section. Consider two PWS systems of form \eqref{pws1}
and let $\chi(x,y)$ and $\widetilde\chi(x,y)$ denote the corresponding
piecewise-smooth vector fields. Further, the upper and lower
sub vector fields are written as follows.
\begin{equation*}
\begin{aligned}
&\chi^\pm(x,y):=\left(\chi^\pm_1(x,y),\chi^\pm_2(x,y)\right)^\top\in C^\infty\left(\mathbb{R}^2,\mathbb{R}^2\right),\\
&\widetilde\chi^\pm(x,y):=\left(\widetilde\chi^\pm_1(x,y),\widetilde\chi^\pm_2(x,y)\right)^\top\in C^\infty\left(\mathbb{R}^2,\mathbb{R}^2\right).
\end{aligned}
\end{equation*}
The distance of two PWS systems with $\chi(x,y)$ and $\widetilde\chi(x,y)$
are defined as
\begin{eqnarray}
\rho\left(\chi,\widetilde\chi\right):=d\left(\chi^+_1,\widetilde\chi^+_1\right)+d\left(\chi^+_2,\widetilde\chi^+_2\right)+d\left(\chi^-_1,\widetilde\chi^-_1\right)+d\left(\chi^-_2,\widetilde\chi^-_2\right),
\label{dist}
\end{eqnarray}
where
\begin{equation*}
d\left(X, Y\right):=\max_{(x,y)\in\bar{\cal U}}\left\{\sum_{i_1+i_2=0}^{1}\left|\frac{\partial^{i_1+i_2} (X-Y)}{\partial x^{i_1}\partial y^{i_2}}\right|\right\}
\end{equation*}
for $X(x,y),Y(x,y)\in C^\infty\left(\mathbb{R}^2,\mathbb{R}\right)$.

\begin{prop}
If $\psi^{\pm}(x,{\boldsymbol l}^{\pm})$ take the form of
$\psi(x,{\boldsymbol l})$ in \eqref{df-psi}, then for any $\epsilon>0$
there exists $\delta>0$ such that
\begin{equation*}
\rho({\cal Z},\widetilde {\cal Z})<\epsilon,~~~~\forall |{\boldsymbol l}^{\pm}|<\delta,
\end{equation*}
where ${\cal Z}(x,y)$ and $\widetilde {\cal Z}\left(x,y\right)$
denote the vector fields of system~\eqref{pws1} and \eqref{pws1-exunfold} respectively.
\label{prop-2}
\end{prop}
The proof of Proposition~\ref{prop-2} is not complicated and thus,
we only provide a brief proof. In fact, it is not hard to obtain that
$\psi^\pm(x,{\boldsymbol l}^\pm)$ and $\dot\psi^\pm(x,{\boldsymbol l}^\pm)$
uniformly converge to the zero function as ${\boldsymbol l}^\pm\to0$
because $\psi^\pm(x,{\boldsymbol l}^\pm)$ is the finite sum of
$H^\pm_i(x,{\boldsymbol l}^\pm_i)$ and Proposition~\ref{prop-2}(b) implies that
$H^\pm_i(x,{\boldsymbol l}^\pm_i)$ and $\dot H^\pm_i(x,{\boldsymbol l}^\pm_i)$
uniformly converge to the zero function as ${\boldsymbol l}^\pm_i\to0$.
Meanwhile, it can be obtained by the definition of $\rho$ in \eqref{dist}
that
\begin{equation*}
\begin{aligned}
\rho({\cal Z},\widetilde {\cal Z})
& = \max_{(x,y)\in \bar{\cal U}}\left\{\left|\psi^+(x,{\boldsymbol l}^+)\right|+\left|\dot\psi^+(x,{\boldsymbol l}^+)\right|\right\}+\max_{(x,y)\in \bar{\cal U}}\left\{\left|\psi^-(x,{\boldsymbol l}^-)\right|+\left|\dot\psi^-(x,{\boldsymbol l}^-)\right|\right\}\\
& \le \sum_{i=1}^{s^+}\left(\left|l^+_{i,3}d^+_{i}\right|+\left|l^+_{i,3}\right|\right)
+\sum_{i=1}^{s^-}\left(\left|l^-_{i,3}d^-_i\right|+\left|l^-_{i,3}\right|\right)
\end{aligned}
\end{equation*}
and further, $\rho({\cal Z},\widetilde{\cal Z})\to0$ as
${\boldsymbol l}^\pm\to{\boldsymbol 0}$.

The appearance of nonequivalent phase portraits under the small variation
of vector fields of PWS system~\eqref{pws1} is called {\it bifurcation},
where the ``nonequivalent'' is the opposite conception of
{\it $\Sigma$-equivalent} in \cite{Teixeira11} and the small variation is
under the meaning of distance $\rho$. Proposition~\ref{prop-2}
ensures that system~\eqref{pws1-exunfold} lies in a small neighborhood of
system~\eqref{pws1}, i.e., under the framework of bifurcation.

\section{Preliminary lemmas}
\setcounter{equation}{0}
\setcounter{lm}{0}
\setcounter{thm}{0}
\setcounter{rmk}{0}
\setcounter{df}{0}
\setcounter{cor}{0}

In order to prove Theorems~\ref{thm1} and \ref{thm2}, we need to analyze
dynamical behaviors of the tangent points of any multiplicity and
the limit cycles. In this section, we introduce the bifurcations of
tangent points and transition maps.

As discussions in section 1, the bifurcations of tangent points
are investigated in \cite{CF25} by the following lemma.
\begin{lm}
Assume that the origin $O$ is a visible tangent point of the upper subsystem of
\eqref{pws1} and $m^+\ge 1$. Then for each $\ell\in\left\{1,...,(m^++1)/2\right\}$,
there exist perturbations of the upper subsystem of \eqref{pws1} such that
there is an orbit exactly passing through $\ell$ visible tangent points.
\label{lm-TP}
\end{lm}
This lemma is proved in \cite{CF25} by investigating the following
functional perturbation of the upper subsystem
\begin{equation*}
        \left( \begin{array}{c}
        \dot{x}\\
        \dot{y}\\
    \end{array} \right)
    =
        \left( \begin{array}{c}
            f^+(x,y+\zeta(x))\\
            \phi^+(x,y+\zeta(x))\prod_{i=1}^{m^+}(x-\lambda^+_i)+\Upsilon^+(x,y+\zeta(x))-\dot\zeta f^+(x,y+\zeta(x))
        \end{array} \right),
\end{equation*}
where $\zeta(x)\in C^{\infty}(\mathbb R)$. Here the parameter
${\boldsymbol \lambda}^+:=(\lambda^+_i)$ ($i=1,...,m^+$)
is introduced to desingularize the tangent point $O$, causing it to split
into several distinct tangent points of multiplicity $1$. And the function
$\zeta(x)$ is introduced to control the relative positions between
tangent orbits and $\Sigma$. Then Lemma~\ref{lm-TP} is proved by taking
appropriate ${\boldsymbol \lambda}^+$ and $\zeta(x)$. But this unfolding
system can not determine the precise positions where the tangent orbit
is tangent with $\Sigma$. This leads to that for the bifurcations of
grazing loops, the local recurrence obtained by this way is not
clear enough to analyze the Poincar\'e return map.
In this paper, a new functional perturbation is proposed to address
these deficiencies.
\begin{proof}[Proof of Lemma~\ref{lm-TP}]
It is not difficult to check that the upper subsystem of \eqref{pws1} take form
\begin{equation}
        \left( \begin{array}{c}
        \dot{x}\\
        \dot{y}\\
    \end{array} \right)
    =
       \left( \begin{array}{c}
            f^+(x,y)\\
            \phi^+(x,y)x^{m^+}+\Upsilon^+(x,y)
        \end{array} \right),
    \label{pws2-up}
\end{equation}
where $m^+\ge 1$, $f^+(0,0)\ne 0, \phi^+(0,0)\ne 0$ and
$\Upsilon^+(x,0)\equiv 0$. Then we consider the functional perturbation
system of \eqref{pws2-up} as follows.
\begin{equation}
        \left( \begin{array}{c}
        \dot{x}\\
        \dot{y}\\
    \end{array} \right)
    =
        \left( \begin{array}{c}
            f^+(x,y)\\
            \phi^+(x,y)x^{m^+}+\Upsilon^+(x,y)+{\cal Q}^+(x,y),
        \end{array} \right)
\label{pws2-up-exunfold}
\end{equation}
where
\begin{equation}
{\cal Q}^+(x,y):=\phi^+(x,y)\prod_{i=1}^{m^+}(x-\lambda^+_i)-\phi^+(x,y)x^{m^+}+\psi^+(x,{\boldsymbol l}^+)
\label{fun-Q}
\end{equation}
and ${\boldsymbol \lambda}^+:=(\lambda^+_i)\in{\mathbb R}^{m^+}$ and $\psi^+(x,{\boldsymbol l}^+)$
takes form of $\psi(x,{\boldsymbol l})$ in \eqref{df-psi}. It is not
hard to check that system~\eqref{pws2-up-exunfold} reads
\begin{equation*}
        \left( \begin{array}{c}
        \dot{x}\\
        \dot{y}\\
    \end{array} \right)
    =
        \left( \begin{array}{c}
            f^+(x,y)\\
            \phi^+(x,y)\prod_{i=1}^{m^+}(x-\lambda^+_i)+\Upsilon^+(x,y)+\psi^+(x,{\boldsymbol l}^+)
        \end{array} \right)
\end{equation*}
by substituting ${\cal Q}^+(x,y)$ into the system~\eqref{pws2-up-exunfold}.

The cases that $m^+=1$ can be proved by taking any $\lambda^+_1\ne0$
and ${\boldsymbol l}^+={\boldsymbol 0}$ in system~\eqref{pws2-up-exunfold}.
For the cases that $m^+\ge3$, we first take
$\lambda^+_i=i\delta$ and ${\boldsymbol l}^+={\boldsymbol 0}$
for sufficiently small $\delta>0$. Thus, there are exactly $m^+$ tangent points
$(\lambda^+_i,0)$ of multiplicities $m^+_i=1$ and $(\lambda^+_{2i-1},0)$
($i=1,...,(m^++1)/2$) are visible. Let $\Theta(x)$ denote the solution of
the Cauchy problem
\begin{equation*}
\frac{dy}{dx}=\frac{\phi^+(x,y)\prod_{i=1}^{m^+}(x-\lambda^+_i)+\Upsilon^+(x,y)}{f^+(x,y)},~~~~
y\left(0\right)=\delta^{m^+}.
\end{equation*}
For all $x\in[0, (m^++1)\delta]$,
\begin{equation}
    \Theta(x) = \delta^{m^+}+\int_{0}^{x}\frac{\phi^+(s,\Theta(s))\prod_{i=1}^{m^+}(s-\lambda^+_i)}{f^+(s,\Theta(s))}ds+\int_{0}^{x}\frac{\Upsilon^+(s,\Theta(s))}{f^+(s,\Theta(s))}ds.
\label{Theta-Esti-1}
\end{equation}
Consider the second term of the right-hand side of \eqref{Theta-Esti-1},
we can obtain that for sufficiently small $\delta>0$,
\begin{equation}
\begin{aligned}
\int_{0}^{x}\frac{\phi^+(s,\Theta(s))\prod_{i=1}^{m^+}(s-\lambda^+_i)}{f^+(s,\Theta(s))}ds
& = x\frac{\phi^+(s_0,\Theta(s_0))\prod_{i=1}^{m^+}(s_0-\lambda^+_i)}{f^+(s_0,\Theta(s_0))}\\
& = \delta^{m^++1}\frac{x^*\phi^+\left(s^*_0\delta,\Theta(s^*_0\delta)\right)}{f^+(s^*_0\delta,\Theta(s^*_0\delta))}\prod_{i=1}^{m^+}(s^*_0-i)\\
& = o(\delta^{m^+}),
\end{aligned}
\label{Theta-Esti-1-1}
\end{equation}
where $x^*:=x/\delta\in[0,(m^++1)]$ and $s^*_0:=s_0/\delta$.
By a similar way, we can obtain that
\begin{equation}
\begin{aligned}
\int_{0}^{x}\frac{\Upsilon^+(s,\Theta(s))}{f^+(s,\Theta(s))}ds
& = x\frac{\Upsilon^+(s_0,\Theta(s_0))}{f^+(s_0,\Theta(s_0))}\\
& = \delta\frac{x^*\widetilde \Upsilon^+(s^*_0\delta,\Theta(s^*_0\delta))}{f^+(s^*_0\delta,\Theta(s^*_0\delta))}\Theta(s^*_0\delta)\\
& = \delta\frac{x^*\widetilde \Upsilon^+(s^*_0\delta,\Theta(s^*_0\delta))}{f^+(s^*_0\delta,\Theta(s^*_0\delta))}\left\{\delta^{m^+} +\sum_{i=1}^{m^+} b_i(\tilde \delta)(s^*_0\delta)^i+o(\delta^{m^+})\right\}\\
& = o(\delta^{m^++1})+\frac{x^*\widetilde \Upsilon^+(s^*_0\delta,\Theta(s^*_0\delta))}{f^+(s^*_0\delta,\Theta(s^*_0\delta))}\left\{\sum_{i=1}^{m^+} b_i(\tilde \delta)(s^*_0\delta)^i\delta\right\},
\end{aligned}
\label{Theta-Esti-1-2}
\end{equation}
where $\tilde \delta=\delta^{m^+}$, $\widetilde\Upsilon^+(x,y):=\Upsilon^+(x,y)/y$ is $C^{\infty}$.
We take $x=(m^++1)\delta/2$ into \eqref{Theta-Esti-1} and
then by \eqref{Theta-Esti-1-1}, \eqref{Theta-Esti-1-2},
\begin{equation}
\sum_{i=1}^{m^+} b_i(\tilde \delta)\frac{(m^++1)^i\delta^i}{2^i}\!+\!o(\delta^{m^+})\!=\!
o(\delta^{m^+})\!+\!\frac{x^*\widetilde \Upsilon^+(s^*_0\delta,\Theta(s^*_0\delta))}{f^+(s^*_0\delta,\Theta(s^*_0\delta))}\!\left\{\sum_{i=1}^{m^+} b_i(\tilde \delta)(s^*_0\delta)^i\delta\!\right\},
\label{Theta-Esti-1-3}
\end{equation}
where $x^*=(m^++1)/2$ and $s^*\in [0,x^*]$.
In fact, $b_1(0)=...=b_{m^+}(0)=0$. The left of \eqref{Theta-Esti-1-3} is
$O(\delta)$ but the right one is $o(\delta)$
if $b_1(0)\ne 0$, i.e., a contradiction. Similarly,
we can prove that $b_2(0)=...=b_{m^+}(0)=0$ and obtain that
\begin{equation}
\Theta(x)=\delta^{m^+}+o(\delta^{m^+})>0
\label{Theta-Esti}
\end{equation}
over $[0,(m^++1)\delta]$.

Then we prove that there exists nonzero function $\psi^+(x,{\boldsymbol l}^+)$
such that there is an orbit passing through $\ell$ visible tangent points.
We begin with the cases that $m^+=3$ and corresponding $\ell\in\left\{1,2\right\}$.
The function $\psi^+(x,{\boldsymbol l}^+)$ is defined by taking $s^+=4$,
\begin{equation*}
\begin{aligned}
&l^+_{1,1}=0,~~~~~~l^+_{1,2}=\delta,&&l^+_{2,1}=\delta,~~~~~l^+_{2,2}=2\delta,\\
&l^+_{3,1}=2\delta,~~~~~l^+_{3,2}=3\delta,&&l^+_{4,1}=3\delta,~~~~l^+_{4,2}=4\delta
\end{aligned}
\end{equation*}
and $l_{i,3}$ ($i=1,...,4$) are undetermined.
Let $\widetilde\Theta(x)$ denote the solution of the following Cauchy problem
\begin{equation*}
\frac{dy}{dx}=\frac{\phi^+(x,y)\prod_{i=1}^{m^+}(x-\lambda^+_i)+\Upsilon^+(x,y)}{f^+(x,y)}+\frac{\psi^+(x,{\boldsymbol l}^+)}{f^+(x,y)},~~~~
y\left(0\right)=\delta^{m^+}.
\end{equation*}
Let $a(x):=d\widetilde\Theta(x)/dl^+_{1,3}$ and a straight computation
shows that for $x=\lambda^+_1$
\begin{equation*}
\begin{aligned}
a(\lambda^+_1)
& = \int_{0}^{\lambda^+_1}\frac{\phi^+_yf^+-f^+_y\phi^+}{(f^+)^2}\prod_{i=1}^{m^+}(s-\lambda^+_i)a(s)ds+\int_{0}^{\lambda^+_1}\frac{\Upsilon^+_yf^+-f^+_y\Upsilon^+}{(f^+)^2}a(s)ds\\
&~~~+\int_{0}^{\lambda^+_1}\frac{\psi^+_{l^+_{1,3}}f^+-f^+_y\psi^+a(s)}{(f^+)^2}ds.
\end{aligned}
\end{equation*}
Taking $l^+_{1,3}\to 0$ in the above equality, we get that there exist
$s_1, s_2\in(0,\lambda^+_1)=(0,\delta)$ such that
$
a(\lambda^+_1)=M_1(\delta)\delta^{m^++1}a(s_1)+M_2(\delta)\delta a(s_2)+M_3(\delta)\delta^2,
$
where $M_i(\delta)$ ($i=1,...,3$) are smooth with respect to $\delta$
and $M_1(0)\ne0$, $M_3(0)\ne0$, $M_2(0)=0$ (resp. $\ne0$) if
$\phi^+_y(0,0)=0$ (resp. $\ne0$). Thus, we obtain that there exists some
$N(\delta)$ satisfying $N(0)\ne0$ such that $a(\lambda^+_1)=N(\delta)\delta^2$
and consequently,
$$\widetilde\Theta(\lambda^+_1)=\Theta(\lambda^+_1)+N(\delta)\delta^2l^+_{1,3}+O((l^+_{1,3})^2).$$
Then by \eqref{Theta-Esti}, there is some $l^+_{1,3}=O(\delta^{m^+-2})$ such that
$\widetilde\Theta(\lambda^+_1)=0$, i.e., the orbit of the upper
subsystem of \eqref{pws2-up-exunfold} with the initial value $(0,\delta^{m^+})$
passes through the tangent point $(\lambda^+_1,0)$. Further, based on this
$l^+_{1,3}$ and over the interval $[\lambda^+_1,\lambda^+_2]$,
it can be proved by a similar way that there exists some $l^+_{2,3}$
such that $\widetilde\Theta(\lambda^+_2)=\Theta(\lambda^+_2)$.
Moreover, over the intervals $[\lambda^+_2,\lambda^+_3]$ and $[\lambda^+_3,4\delta]$,
there exists some $l^+_{3,3}$ and $l^+_{4,3}$ such that
$\widetilde\Theta(\lambda^+_3)=0$ and $\widetilde\Theta(4\delta)=\Theta(4\delta)$.
Analysis above implies that system~\eqref{pws2-up-exunfold} has an orbit passing
through the visible bifurcating tangent points $(\lambda^+_1,0)$ and $(\lambda^+_3,0)$
as shown in Figure~\ref{Fig-TPB}(a), which completes the proof of $\ell=2$.
\begin{figure}[h]
\centering
\subfigure[$\widetilde\Theta(\lambda^+_3)=0$]
 {
  \scalebox{0.4}[0.4]{
   \includegraphics{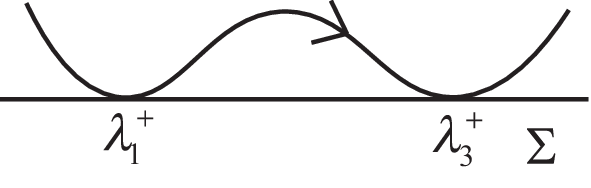}}}~~~
\subfigure[$\widetilde\Theta(\lambda^+_3)>0$]
 {
  \scalebox{0.4}[0.4]{
   \includegraphics{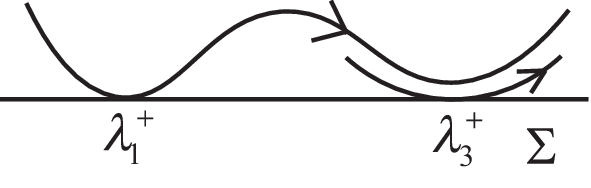}}}~~~
\subfigure[$\widetilde\Theta(\lambda^+_3)<0$]
 {
  \scalebox{0.4}[0.4]{
   \includegraphics{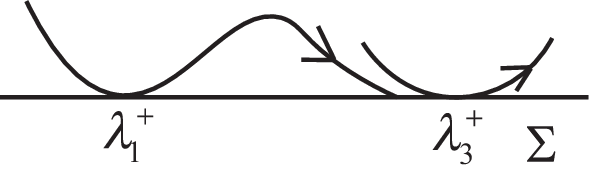}}}
   \caption{orbit of system \eqref{pws2-up-exunfold} passing through tangent points}
\label{Fig-TPB}
\end{figure}

The proof of $m^+\ge5$ is similar and we omit the statements.
\end{proof}
Similarly, one can check that the parameter ${\boldsymbol \lambda}^+$ is
introduced to desingularize the tangent point $O$, causing it to split
into several distinct tangent points of multiplicity $1$ at $(\lambda^+_i,0)$.
Meanwhile, it is also proved that by taking $\lambda^+_i=i\delta$ ($i=1,...,m^+$)
for small $\delta>0$, there are some orbits in the region 
$\left\{(x,y):~x\in(0,(m^++1)\delta),~y\in(0,O(\delta^{m^+}))\right\}$,
which allows us to define the function $\psi^+(x,{\boldsymbol l}^+)$ 
to control the relative positions between tangent orbits and $\Sigma$.
That is, there is a tangent orbit being tangent with $\Sigma$
exactly at some $(\lambda^+_i,0)$. It is the basic to define and analyze the
Poincar\'e return map near the grazing loop, which is shown in the next 
section. On the other hand, the form of functional perturbation
system~\eqref{pws2-up-exunfold} facilitates the analysis
of the relationship between local and global recurrences, which is
also shown in the next section.

In order to characterize the global recurrences, we introduce the
{\it transition map}. Consider the following smooth system
\begin{equation}
        \left( \begin{array}{c}
        \dot{x}\\
        \dot{y}\\
    \end{array} \right)
    =
    \left( \begin{array}{c}
            f(x,y)\\
            g(x,y)
    \end{array} \right)
    :={\cal Z}(x,y),
\label{sys-transition}
\end{equation}
where $(x,y)\in{\mathbb R}^2$ and $f, g\in C^{\infty}({\mathbb R}^2)$.
Let
$
\gamma(t,x_0,y_0):=\left(\gamma_1(t,x_0,y_0),\gamma_2(t,x_0,y_0)\right)^\top
$
denote the orbit of system \eqref{sys-transition} with the initial value $(x_0,y_0)$.
For a given $T>0$ and a point $(x_0,y_0)$ satisfying ${\cal Z}(x_0,y_0)\ne 0$,
denote the point $\gamma(T,x_0,y_0)$ as $(x_1,y_1)^\top$.
Then we take line segment $S_0$ and $S_1$ at $(x_0,y_0)$ and $(x_1,y_1)$, respectively.
Moreover, the line segment $S_1$ is required to be transversal to the orbits
of system~\eqref{sys-transition}. Further, for the point
$(\widetilde x_0,\widetilde y_0)\in S_0$ sufficiently close to $(x_0,y_0)$,
the orbit $\gamma(t,\widetilde x_0,\widetilde y_0)$ transversally intersects
$S_1$ at some point $(\widetilde x_1,\widetilde y_1)$ which is sufficiently closed
to $(x_1,y_1)$. Let $N_0:=(N_{01},N_{02})^\top$ and $N_1:=(N_{11},N_{12})^\top$
denote the unit vector parallel to $S_0$ and $S_1$ respectively,  and write
\begin{equation*}
(\widetilde x_0,\widetilde y_0)=(x_0+rN_{01},y_0+rN_{02}),~~~~~~~
(\widetilde x_1,\widetilde y_1)=(x_1+V(r)N_{11},y_1+V(r)N_{12})
\end{equation*}
for sufficiently small $r\ne0$ as shown in Figure~\ref{Fig-trans-1}(a).

\begin{figure}[h]
\centering
\subfigure[transversal case]
 {
  \scalebox{0.35}[0.35]{
   \includegraphics{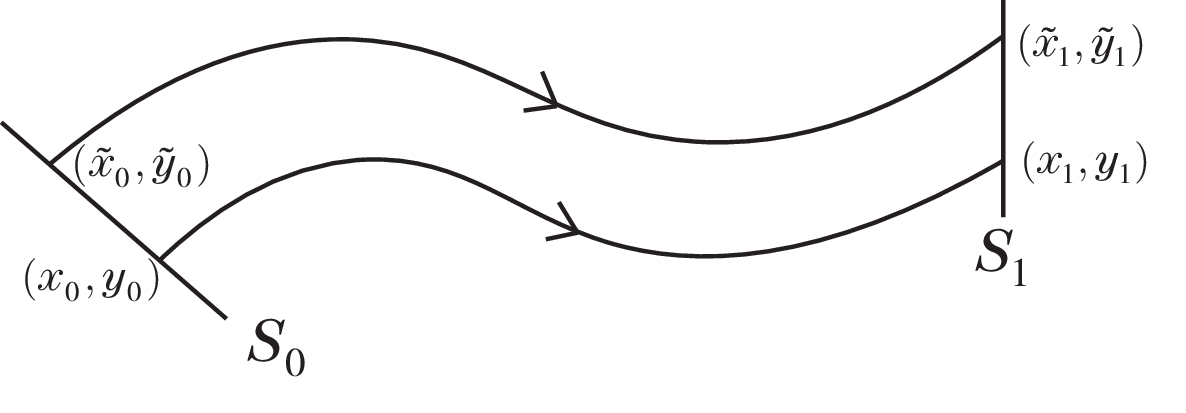}}}~
\subfigure[tangent case]
 {
  \scalebox{0.35}[0.35]{
   \includegraphics{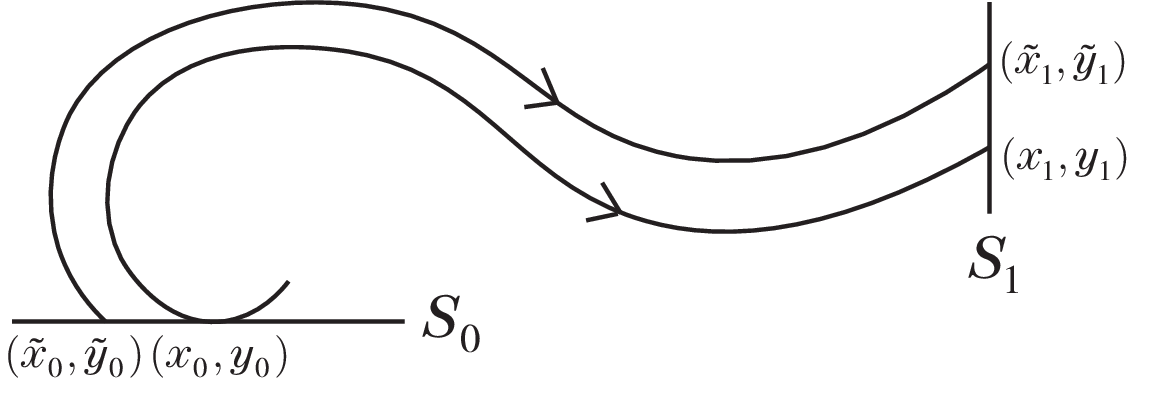}}}
   \caption{The transition map $V(r)$}
\label{Fig-trans-1}
\end{figure}

As indicated in \cite{ZhangJ98}, function $V(r)$ is called the transition map
of system~\eqref{sys-transition}. It is generally used to analyze limit cycles,
homoclinic loops and heteroclinic loops. The transition $V(r)$ is $C^\infty$
with respect to $r$ and for sufficiently small $r>0$
\begin{equation}
V(r)=V_1r+O(r^2),
\label{V-Expan}
\end{equation}
where
\begin{equation*}
V_1:=\frac{\Delta_0}{\Delta_1}\exp\left\{\int_0^{T}\frac{\partial f\left(\gamma\left(s,x_0,y_0\right)\right)}{\partial x}+\frac{\partial g\left(\gamma\left(s,x_0,y_0\right)\right)}{\partial y}ds\right\}\ne0
\end{equation*}
and
\begin{eqnarray}
\Delta_0:=\det\left({\cal Z}(x_0,y_0),N_0\right), ~~~~\Delta_1:=\det\left({\cal Z}(x_1,y_1),N_1\right).
\label{D0D1}
\end{eqnarray} (see, e.g., \cite{JKHaleBook,KuznetsovBook,ZhangJ98}).
It is not difficult to check that the transversality between
$S_1$ and orbits ensures the non-vanishing denominator $\Delta_1$ in \eqref{D0D1}.
Moreover, the transition map can be used to characterize a standard limit cycle
by taking $S_0=S_1$ and consequently, $V_1=\sigma$ because $\Delta_0=\Delta_1$.
Further, $\sigma<0$ (resp. $\sigma>0$) implies that limit cycle $L_I$ is stable (resp. unstable).

Finally, we consider the transition map for the cases
that orbit $\gamma(t,x_0,y_0)$ is tangent with $S_0$ as shown in
Figure~\ref{Fig-trans-1}(b). It is not difficult to check that $V_1=0$ because of
vanishing $\Delta_0$ and consequently, the transition map $V(r)$ starts at
least degree $2$. The following lemma characterizes the corresponding
Taylor expand, which is proved in \cite{Ponce15,Han13,Novaes18} for $m=1$ and
in \cite{Andrade23,Jeffrey19} for general $m$.

\begin{lm} Assume that orbit  $\gamma\left(t,x_0,y_0\right)$ of
system~\eqref{sys-transition} is tangent with horizontal line $S_0$ at $(x_0,y_0)$
and transversally intersects $S_1$ at $(x_1,y_1)$ at time $t=T$ as shown in
{\rm Figure~\ref{Fig-trans-1}(b)}. If there exists an integer $m\ge 1$ such that
\begin{equation*}
g(x_0,y_0)=\frac{\partial g}{\partial x}(x_0,y_0)=...=\frac{\partial^{m-1} g}{\partial x^{m-1}}(x_0,y_0)=0,~~\frac{\partial^{m} g}{\partial x^{m}}(x_0,y_0)\ne0,
\end{equation*}
then the transition map
$
V(r)=V_{m+1}r^{m+1}+O\left(r^{m+2}\right),
$
where
\begin{equation*}
V_{m+1}:=\frac{\partial^m g}{\partial x^m}(x_0,y_0)\frac{N^{m+1}_{01}}{\left(m+1\right)!\Delta_1}\exp\left\{\int_0^{T}\frac{\partial f\left(\gamma\left(s,x_0,y_0\right)\right)}{\partial x}
+\frac{\partial g\left(\gamma\left(s,x_0,y_0\right)\right)}{\partial y}ds\right\}\ne0
\end{equation*}
and $N_{01}$ is the first component of unit vector $N_0$ parallel to $S_0$ and
$\Delta_1$ is defined in \eqref{D0D1}.
\label{lm-TM}
\end{lm}

\section{Proof of main results}
\setcounter{equation}{0}
\setcounter{lm}{0}
\setcounter{thm}{0}
\setcounter{rmk}{0}
\setcounter{df}{0}
\setcounter{cor}{0}

For the system~\eqref{pws1} satisfying assumptions {\bf A}$_1$ and {\bf A}$_2$,
it takes the following form
\begin{equation}
\begin{aligned}
            \left(\begin{array}{c}
            \dot{x}\\
            \dot{y}\\
        \end{array}\right)=\left\{
        \begin{aligned}
    &       \left(\begin{array}{c}
                f^+(x,y)\\
                \phi^+(x,y)x^{2k+1}+\Upsilon^+(x,y)
            \end{array}\right)   &&\mathrm{if~}(x,y)\in\Sigma^+,\\
    &       \left(\begin{array}{c}
                f^-(x,y)\\
                g^-(x,y)
            \end{array}\right)   &&\mathrm{if~}(x,y)\in\Sigma^-,
        \end{aligned} \right.
    \end{aligned}
\label{pws-thm1}
\end{equation}
where $k\in{\mathbb Z}^+$, $\phi^+(0,0)\ne0$ and $\Upsilon^+(x,0)\equiv0$.
The proof of Theorem~\ref{thm1} is given by analyzing the following
perturbation system
\begin{equation}
\begin{aligned}
            \left(\begin{array}{c}
            \dot{x}\\
            \dot{y}\\
        \end{array}\right)\!=\!\left\{
        \begin{aligned}
    &       \left(\begin{array}{c}
                f^+(x,y)\\
                \phi^+(x,y)x^{2k+1}+\Upsilon^+(x,y)+{\cal P}^+(x,y)
            \end{array}\right)   &&\mathrm{if}~(x,y)\in\Sigma^+,\\
    &       \left(\begin{array}{c}
                f^-(x,y)\\
                g^-(x,y)
            \end{array}\right)   &&\mathrm{if}~(x,y)\in\Sigma^-,
        \end{aligned} \right.
    \end{aligned}
\label{pws-thm1-pf}
\end{equation}
where
\begin{equation}
{\cal P}^+(x,y):=\left(\psi^+(x,{\boldsymbol l}^+)+\phi^+(x,y)\!\!\prod\limits_{i=1}^{2k+1}\!\!(x-\lambda^+_i)-\phi^+(x,y)x^{2k+1}\right)\nu(x)\mu(y)
\label{perturb-fun}
\end{equation}
for
${\boldsymbol \lambda}^+:=(\lambda^+_i)\in{\mathbb R}^{2k+1}$,
$\psi^+(x,{\boldsymbol l}^+)$ is of form $\psi(x,{\boldsymbol l})$
in \eqref{df-psi},
\begin{equation}
    \nu(x):=\left\{
    \begin{aligned}
    &0,&&x\in(-\infty,k_1]\cup[k_4,+\infty),\\
    &\frac{e^{-\frac{1}{x-k_1}}}{e^{-\frac{1}{x-k_1}}+e^{\frac{1}{x-k_2}}},&&x\in(k_1,k_2),\\
    &1,&&x\in[k_2,k_3],\\
    &\frac{e^{\frac{1}{x-k_3}}}{e^{-\frac{1}{x-k_3}}+e^{\frac{1}{x-k_4}}},&&x\in(k_3,k_4)\\
    \end{aligned}
    \right.
\label{df-nu}
\end{equation}
for some $k_1<k_2<k_3<k_4$ and
\begin{equation}
    \mu(y):=\left\{
    \begin{aligned}
    &1,&&y\in(-\infty,r_1],\\
    &\frac{e^{\frac{1}{y-r_2}}}{e^{-\frac{1}{y-r_1}}+e^{\frac{1}{y-r_2}}},&&y\in(r_1,r_2),\\
    &0,&&y\in[r_2,+\infty]\\
    \end{aligned}
    \right.
\label{df-mu}
\end{equation}
for some $r_1<r_2$. 
\begin{proof}[Proof of Theorem~\ref{thm1}]
We begin with the cases that $g^-(0,0)>0$ and $\sigma<0$, i.e., 
the grazing loop $L_{I}$ is of the configuration in Figure~\ref{Fig-Gra}(a) 
and is a stable limit cycle. The proof of $k=1$ proceeds the following
three steps.

{\it Step 1. Determine appropriate functions $\nu(x)$ and $\mu(y)$.}
In what follows, the function $\nu(x)$ is defined by taking
$k_1=-3\delta, k_2=-2\delta, k_3=2\delta, k_4=3\delta$
and the function $\mu(y)$ is defined by taking
$r_1=\delta, r_2=\delta+c$
for a sufficiently small $\delta>0$,
where $c$ is a positive constant such that the region
$\{(x,y):~x\in(-3\delta,3\delta),y\in(\delta,\delta+c)\}$ is in the
interior of $L_{I}$. According to the definitions \eqref{df-nu} and \eqref{df-mu},
it is not hard to check that in the region
$\{(x,y):~x\in(-2\delta,2\delta),y\in(0,\delta)\}$
\begin{equation*}
{\cal P}^+(x,y)=\phi^+(x,y)\prod\limits_{i=1}^{3}(x-\lambda^+_i)-\phi^+(x,y)x^{3}+\psi^+(x,{\boldsymbol l}^+),
\end{equation*}
i.e., the function ${\cal Q}^+(x,y)$ in \eqref{fun-Q} for $m^+=3$.
This implies that
$$\widetilde{\cal Z}^+_2(x,y)=\phi^+(x,y)\prod\limits_{i=1}^{3}(x-\lambda^+_i)+\Upsilon^+(x,y)+\psi^+(x,{\boldsymbol l}^+),$$
where $\widetilde{\cal Z}^+(x,y):=(\widetilde{\cal Z}^+_1(x,y),\widetilde{\cal Z}^+_2(x,y))^\top$
denotes the upper vector field of system~\eqref{pws-thm1-pf}.
On the other hand, outside the region $\Sigma_p:=\{(x,y):~x\in(-3\delta,3\delta),y\in(0,\delta+c)\}$,
we obtain that $\widetilde{\cal Z}^+(x,y)\equiv{\cal Z}^+(x,y)$
because ${\cal P}^+(x,y)\equiv0$, i.e., there is no perturbation. Here
where ${\cal Z}^+(x,y):=({\cal Z}^+_1(x,y),{\cal Z}^+_2(x,y))^\top$ denotes 
the upper vector field of system~\eqref{pws-thm1}.

{\it Step 2. Determine parameter ${\boldsymbol \lambda}^+$ in
system~\eqref{pws-thm1-pf} to desingularize tangent points.}
For system~\eqref{pws-thm1-pf} and sufficiently small $\delta>0$, we
first take
$\lambda^+_1=-\delta, \lambda^+_2=0, \lambda^+_3=\delta$
and $\psi^+(x,{\boldsymbol l}^+)\equiv0$, i.e., obtain a
system in region $\{(x,y):~x\in(-2\delta,2\delta),y\in(-\infty,\delta)\}$
\begin{equation}
\begin{aligned}
            \left(\begin{array}{c}
            \dot{x}\\
            \dot{y}\\
        \end{array}\right)=\left\{
        \begin{aligned}
    &       \left(\begin{array}{c}
                f^+(x,y)\\
                \phi^+(x,y)\prod\limits_{i=1}^{2k+1}(x-\lambda^+_i)+\Upsilon^+(x,y)
            \end{array}\right)   &&\mathrm{if}~(x,y)\in\Sigma^+,\\
    &       \left(\begin{array}{c}
                f^-(x,y)\\
                g^-(x,y)
            \end{array}\right)   &&\mathrm{if}~(x,y)\in\Sigma^-.
        \end{aligned} \right.
    \end{aligned}
\label{tpws-thm1-pf}
\end{equation}
System~\eqref{tpws-thm1-pf}, i.e., system~\eqref{pws-thm1-pf} with
$\psi^+(x,{\boldsymbol l}^+)\equiv0$, is called the {\it transitional system}
in what follows because it is used to establish some relationships
between system~\eqref{pws-thm1} and \eqref{pws-thm1-pf}.
It is not hard to check that $(\lambda^+_1,0)$ and $(\lambda^+_3,0)$ are
visible tangent points and $(\lambda^+_2,0)$ is an invisible tangent point.
Then we take two small vertical line segments
$S(c):=\{(c,y):~y\in(0,2\delta^3)\}$ for $c=\pm3\delta$
and further, the grazing loop $L_{I}$ of unperturbed system~\eqref{pws-thm1}
intersects $S(-3\delta)$ and $S(3\delta)$ at some point
$(-3\delta,{\cal Y}^*)$ and $(3\delta,{\cal Y}_*)$ respectively. Meanwhile,
it can be proved by the analysis in Lemma~\ref{lm-TP} that there
exist nonzero constants $M^*$ and $M_*$ such that 
${\cal Y}^*=M^*\delta^4, {\cal Y}_*=M_*\delta^4$.
Then for point $(-3\delta,\delta^3)$, the orbit $\gamma^+(t,-3\delta,\delta^3)$
sequentially intersects $S(3\delta)$ and $S(-3\delta)$ at some point
$(3\delta,{\cal Y}_1)$ and $(-3\delta,{\cal Y}_2)$ for forward direction
respectively. Via transition map and \eqref{V-Expan}, we obtain that
$${\cal Y}_2={\cal Y}^*+\exp\{\sigma\}(\delta^3-{\cal Y}^*)+O((\delta^3-{\cal Y}^*)^2)=\exp\{\sigma\}\delta^3+O(\delta^4)>0$$
and further, ${\cal Y}_1=\exp\{\sigma\}\delta^3+O(\delta^4)>0$ by the analysis
in Lemma~\ref{lm-TP}.

Let
$\widehat\gamma^+(t,x_0,y_0):=(\widehat\gamma^+_1(t,x_0,y_0),\widehat\gamma^+_2(t,x_0,y_0))^\top$
denote the orbit of the upper subsystem of \eqref{tpws-thm1-pf} with initial
value $(x_0,y_0)$. Since ${\cal P}^+(x,y)\equiv0$ beside the region
$\Sigma_p$, the orbit $\widehat\gamma^+(t,-3\delta,\delta^3)$ intersects $S(3\delta)$
at the point $(3\delta,{\cal Y}_1)$ for forward direction. Then let
$\Theta(x)$, $\widehat\Theta(x)$ denote the solutions of Cauchy problems
\begin{equation*}
\frac{dy}{dx}=\frac{\phi^+(x,y)x^{3}+\Upsilon^+(x,y)}{f^+(x,y)},~~~~
y\left(0\right)={\cal Y}_1,
\end{equation*}
\begin{equation*}
\frac{dy}{dx}=\frac{\phi^+(x,y)x^{3}+\Upsilon^+(x,y)+{\cal P}^+(x,y)}{f^+(x,y)},~~~~
y\left(0\right)={\cal Y}_1,
\end{equation*}
respectively. Further, over the interval $[-3\delta,3\delta]$, we obtain that
\begin{equation*}
\begin{aligned}
\left|\widehat\Theta(x)-\Theta(x)\right|
& = \left|\int_{0}^{x}\frac{\phi^+(s,\widehat\Theta(s))s^{3}+\Upsilon^+(s,\widehat\Theta(s))+{\cal P}^+(s,\widehat\Theta(s))}{f^+(s,\widehat\Theta(s))}ds\right.\\
& ~~~\left.-\int_{0}^{x}\frac{\phi^+(s,\Theta(s))s^{3}+\Upsilon^+(s,\Theta(s))}{f^+(s,\Theta(s))}ds\right|\\
& \le \left|\int_{0}^{x}\frac{\phi^+(s,\widehat\Theta(s))s^{3}+\Upsilon^+(s,\widehat\Theta(s))+{\cal P}^+(s,\widehat\Theta(s))}{f^+(s,\widehat\Theta(s))}ds\right|\\
& ~~~+\left|\int_{0}^{x}\frac{\phi^+(s,\Theta(s))s^{3}+\Upsilon^+(s,\Theta(s))}{f^+(s,\Theta(s))}ds\right|\\
& = O(\delta^4)
\end{aligned}
\end{equation*}
via the analysis in Lemma~\ref{lm-TP}, i.e., $\widehat\Theta(x)=\exp\{\sigma\}\delta^3+O(\delta^4)$ for $x\in[-3\delta,3\delta]$.
Denote the intersections between $\widehat\gamma^+(t,-3\delta,\delta^3)$
and $S(-2\delta)$ for backward direction (resp. $S(2\delta)$ for forward direction)
as $(-2\delta,\widehat{\cal Y}_1)$ (resp. $(2\delta,\widehat{\cal Y}_2)$),
where $S(c):=\{(c,y):~y\in(0,2\delta^3)\}$.

{\it Step 3. Determine appropriate functions $\psi^+(x,{\boldsymbol l}^+)$
in system~\eqref{pws-thm1-pf}.}
The function $\psi^+(x,{\boldsymbol l}^+)$ is defined by taking $s^+=4$,
\begin{equation*}
\begin{aligned}
&l^+_{1,1}=-2\delta,~~~l^+_{1,2}=-\delta,&&l^+_{2,1}=-\delta,~~~l^+_{2,2}=0,\\
&l^+_{3,1}=0,~~~~~~~l^+_{3,2}=\delta,&&l^+_{4,1}=\delta,~~~~~l^+_{4,2}=2\delta
\end{aligned}
\end{equation*}
and $l^+_{i,3}$ ($i=1,...,4$) are undetermined. Let
$\widetilde\gamma^+(t,x_0,y_0)=(\widetilde\gamma^+_1(t,x_0,y_0),\widetilde\gamma^+_2(t,x_0,y_0))^\top$
denote the orbit of the upper subsystem of \eqref{pws-thm1-pf} with initial
value $(x_0,y_0)$. The definition of $\psi^+(x,{\boldsymbol l}^+)$ implies that
$\widetilde\gamma^+(t,x_0,y_0)\equiv\widehat\gamma^+(t,x_0,y_0)$
beside the region $\{(x,y):~x\in(-2\delta,2\delta),y\in(0,\delta+c)\}$
and thus, the orbit $\widetilde\gamma^+(t,-3\delta,\delta^3)$ intersects
with $S(-2\delta)$ (resp. $S(2\delta)$) at the point $(-2\delta,\widehat{\cal Y}_1)$
(resp. $(2\delta,\widehat{\cal Y}_2)$) for backward (resp. forward) direction.

Then $l^+_{i,3}$ ($i=1,...,4$) are determined as follows.
It can be proved by the analysis in Lemma \ref{lm-TP} that
there exist $l^+_{1,3}$ such that orbit $\widetilde\gamma^+(t,-3\delta,\delta^3)$
passes through visible tangent points $(\lambda^+_1,0)$ for backward direction.
Further, based on this $l^+_{1,3}$ there exist some $l^+_{i,3}$ ($i=2,3$)
such that this orbit passes through the visible tangent point $(\lambda^+_3,0)$.
Finally based on these $l^+_{i,3}$ ($i=1,2,3$), there exist some $l^+_{4,3}$
such that this orbit passes through the point $(2\delta, \widehat{\cal Y}_2)$
for backward direction. This implies that there is a grazing loop connecting
two tangent points of multiplicity $(1,0)$ bifurcating from $L_{I}$
as shown in Figure~\ref{Fig-THM1-1}(a). Similarly, there exist $l^+_{i,3}$ ($i=1,...,4$)
such that there is a grazing loop connecting tangent point $(\lambda^+_1,0)$
(resp. $(\lambda^+_3,0)$) bifurcating from $L_{I}$ as shown in
Figure~\ref{Fig-THM1-1}(b) (resp. (c)).
\begin{figure}[h]
\centering
\subfigure[connecting $(\lambda^+_1,0)$, $(\lambda^+_3,0)$]
 {
  \scalebox{0.42}[0.42]{
   \includegraphics{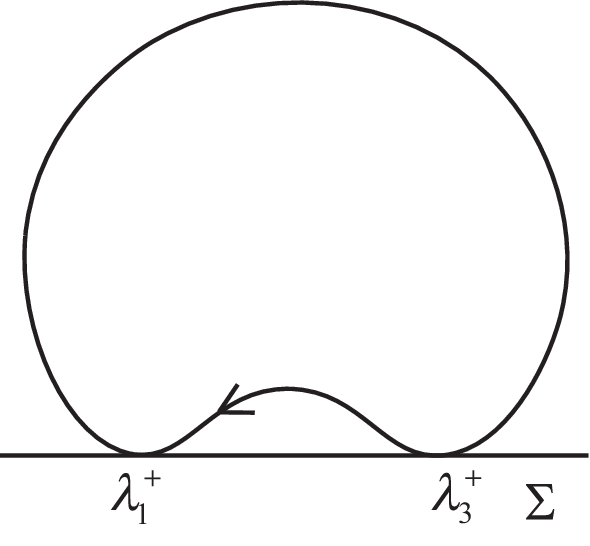}}}~
\subfigure[connecting $(\lambda^+_1,0)$]
 {
  \scalebox{0.42}[0.42]{
   \includegraphics{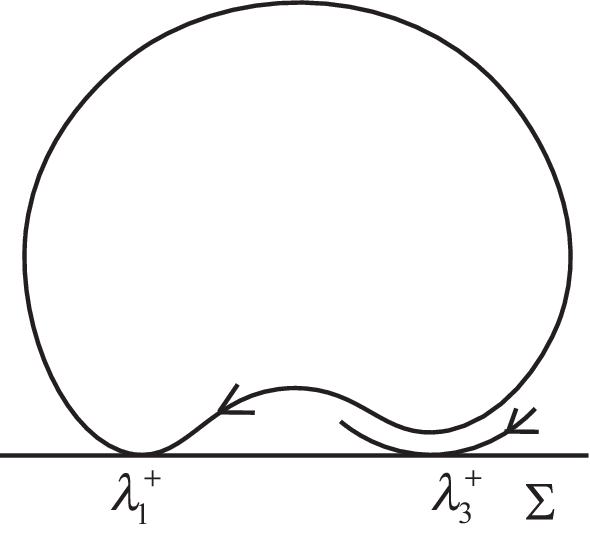}}}~
\subfigure[connecting $(\lambda^+_3,0)$]
 {
  \scalebox{0.42}[0.42]{
   \includegraphics{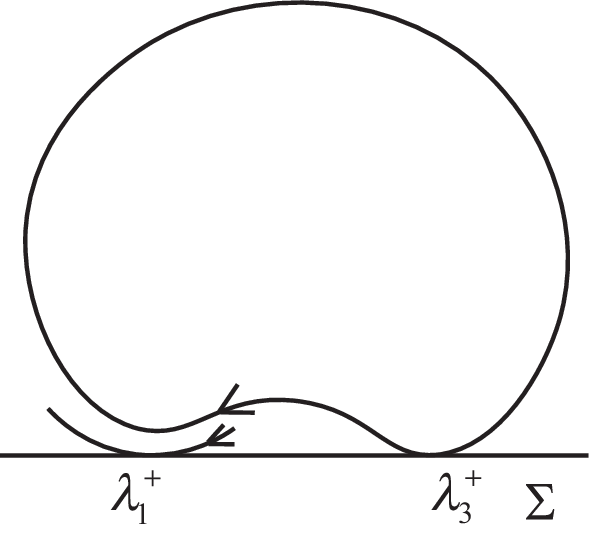}}}
   \caption{The grazing loop of \eqref{pws-thm1-pf} connecting tangent points}
\label{Fig-THM1-1}
\end{figure}

Further, we perturb these grazing loops to obtain the sliding loops
connecting several tangent points. For the cases $\ell=2$, we can perturb
$l^+_{4,3}$ such that the orbits $\widetilde \gamma^+(t,-3\delta,\delta^3)$
intersects $S(2\delta)$ at some point $(2\delta,\widetilde{\cal Y}_2)$ satisfying
$\widetilde {\cal Y}_2>\widehat {\cal Y}_2$ for backward direction.
This implies that the orbit $\widetilde \gamma^+(t,-3\delta,\delta^3)$
intersects the switching manifold $\Sigma$ at some point $(q,0)$
satisfying $q>\lambda^+_3$ for forward direction. It is not hard to check that
$(q,0)\in\Sigma_s$ (the sliding region). As indicated in \cite{Filippov88},
the sliding vector field $X_s(x)$ is defined as
\begin{equation*}
X_s(x):=({\cal X}_s(x),0)^\top=\left(\frac{f^+(x,0)g^-(x,0)-f^-(x,0)g^+(x,0)}{g^-(x,0)-g^+(x,0)},0 \right)^\top
\end{equation*}
and it is not difficult to check that ${\cal X}_{s}(x,0)<0$, i.e.,
the sliding orbit goes to $(\lambda^+_3,0)$. Thus, the orbit
$\widetilde \gamma^+(t,-3\delta,\delta^3)$ goes to the visible tangent point
$(\lambda^+_3,0)$ along a sliding orbit. Further, we continue to perturb
$l^+_{2,3}$ and $l^+_{3,3}$ such that this orbit goes to the visible tangent
point $(\lambda^+_1,0)$ along a sliding orbit by a similar way. Thus, there
is a sliding loop connecting two tangent points as shown in
Figure~\ref{Fig-THM1-2}(a). Note that there is no standard
limit cycle in $\Sigma^+$ because of the stability of $L_{I}$.
The proof of $\ell=2$ is finished.
\begin{figure}[h]
\centering
\subfigure[connecting $(\lambda^+_1,0)$, $(\lambda^+_3,0)$]
 {
  \scalebox{0.42}[0.42]{
   \includegraphics{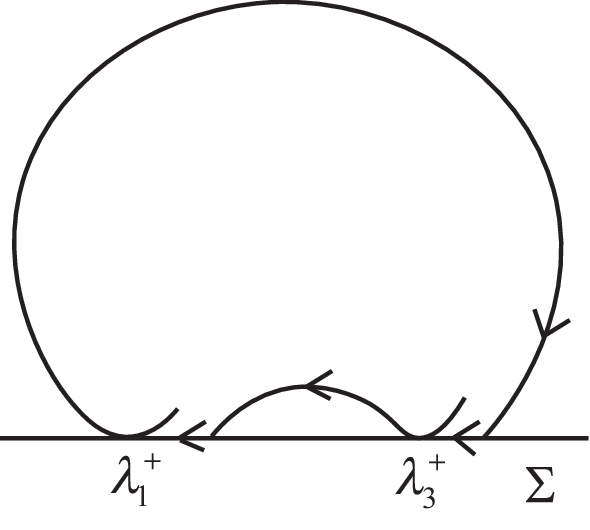}}}~
\subfigure[connecting $(\lambda^+_1,0)$]
 {
\scalebox{0.42}[0.42]{
   \includegraphics{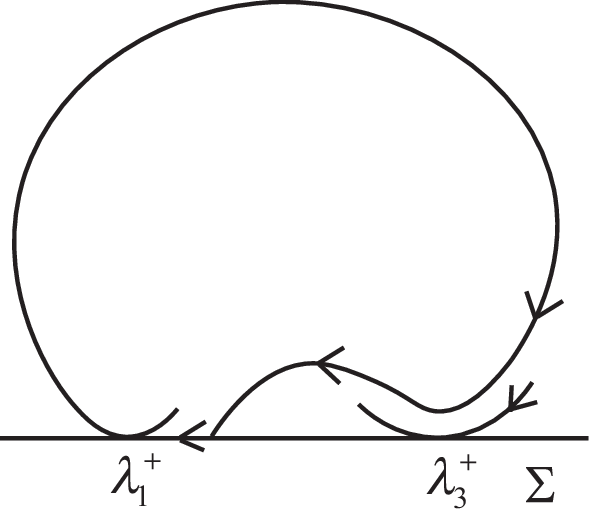}}}~
\subfigure[connecting $(\lambda^+_3,0)$]
 {
\scalebox{0.42}[0.42]{
   \includegraphics{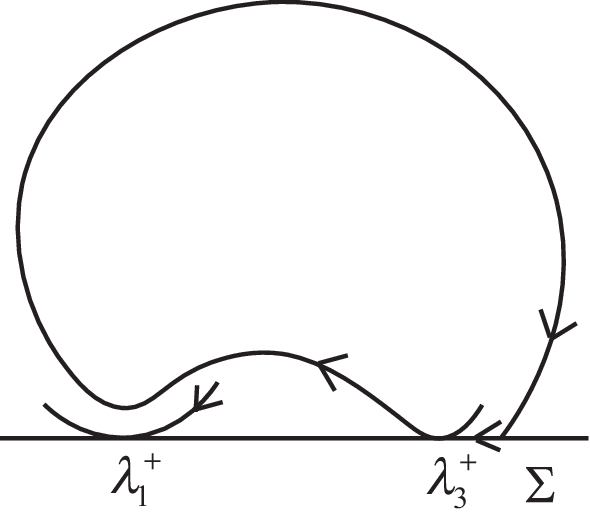}}}
   \caption{The sliding loop of \eqref{pws-thm1-pf} connecting tangent points}
\label{Fig-THM1-2}
\end{figure}

Similarly, for the grazing loop connecting one tangent point $(\lambda^+_1,0)$
(resp. $(\lambda^+_3,0)$), we can perturb some $l^+_{i,3}$ such that
this grazing loop becomes a sliding loop and simultaneously, there is no
standard limit cycle in $\Sigma^+$. Proof of $\ell=1$ is finished and consequently,
we prove the cases that $k=1$. Corresponding examples are shown in Figure~\ref{Fig-THM1-2}(b)
and (c).

Here we claim that for any $\epsilon>0$, there exists $\delta_0>0$ such that
$
\rho({\cal Z},\widetilde {\cal Z})<\epsilon, \forall \delta<\delta_0,
$
i.e., all analysis is under the framework of perturbation,
where ${\cal Z}$ and $\widetilde {\cal Z}$ denote the vector fields
of system~\eqref{pws-thm1} and \eqref{pws-thm1-pf} respectively.
Beside the region $\Sigma_p$, one can check that ${\cal Z}\equiv\widetilde{\cal Z}$
because of the definitions of $\nu(x)$ and $\mu(y)$. In the region
$\{(x,y):~x\in(-2\delta,2\delta),y\in(0,\delta+c)\}$, one can check that
$
\rho({\cal Z},\widehat{\cal Z})
\le \rho({\cal Z},\widehat{\cal Z})+\rho(\widehat{\cal Z},\widetilde{\cal Z})\to 0
$
as $\delta\to0$ via Proposition~\ref{prop-2}. In the region
$\{(x,y):~x\in(-3\delta,-2\delta), y\in(0,\delta+c)\}$, one can compute
that as $\delta\to0$,
\begin{equation*}
\begin{aligned}
|{\cal Z}^+_2-\widetilde{\cal Z}^+_2|
& = \left|\phi^+(x,y)\prod_{i=1}^{3}(x-\lambda^+_i)-\phi^+(x,y)x^{3}\right|\nu(x)\mu(y)\\
& \le |\phi^+(x,y)|\left|\prod_{i=1}^{3}(x-\lambda^+_i)-x^{2k+1}\right|\to 0
\end{aligned}
\end{equation*}
and
\begin{equation*}
\begin{aligned}
|{\cal Z}^+_{2y}-\widetilde{\cal Z}^+_{2y}|
& = \left|\left(\phi^+_y(x,y)\prod_{i=1}^{3}(x-\lambda^+_i)-\phi^+_y(x,y)x^{3}\right)\nu(x)\mu(y)\right.\\
& ~~~+\left.\left(\phi^+(x,y)\prod_{i=1}^{3}(x-\lambda^+_i)-\phi^+(x,y)x^{3}\right)\nu(x)\dot\mu(y)\right|\\
& \le |\phi^+_y(x,y)|\left|\prod_{i=1}^{3}(x-\lambda^+_i)-x^{3}\right|+|\phi^+(x,y)|\left|\prod_{i=1}^{3}(x-\lambda^+_i)-x^{3}\right||\dot\mu(y)|\to0
\end{aligned}
\end{equation*}
because of the boundedness of $\dot\mu(y)$.
Straight computation shows that for $x\in (-3\delta,-2\delta)$
$$
\dot\nu(x)=\left(\frac{1}{(x+3\delta)^2}+\frac{1}{(x+2\delta)^2}\right)
\frac{e^{\eta(x)}}{\left(1+e^{\eta(x)}\right)^2},
$$
where
$
\eta(x):=\left(x+3\delta\right)^{-1}+\left(x+2\delta\right)^{-1}.
$
Thus, we obtain that
\begin{equation*}
\begin{aligned}
|{\cal Z}^+_{2x}-\widetilde{\cal Z}^+_{2x}|
& = \left|\phi^+_x(x,y)\left(\prod_{i=1}^{3}(x-\lambda^+_i)-x^{3}\right)\nu(x)\mu(y)\right.\\
& ~~~+\phi^+(x,y)\left(\sum_{j=1}^{3}\prod_{i=1,i\ne j}^{3}(x-\lambda^+_i)-3x^{2}\right)\nu(x)\mu(y)\\
& ~~~+\left.\phi^+(x,y)\left(\prod_{i=1}^{3}(x-\lambda^+_i)-x^{3}\right)\dot\nu(x)\mu(y)\right|\\
& \le \left|\phi^+_x(x,y)\left(\prod_{i=1}^{3}(x-\lambda^+_i)-x^{3}\right)\right|\\
& ~~~+\left|\phi^+(x,y)\left(\sum_{j=1}^{3}\prod_{i=1,i\ne j}^{3}(x-\lambda^+_i)-3x^{2}\right)\right|\\
& ~~~+\left|\phi^+(x,y)\left(\prod_{i=1}^{3}(x-\lambda^+_i)-x^{3}\right)\dot\nu(x)\right|\\
& = O(\delta^{3})+O(\delta^{2})+O(\delta)\to0
\end{aligned}
\end{equation*}
as $\delta\to0$. This proves the claim.

The cases that $k\ge2$ can be proved by a similar way and we just give
a brief proof. The function $\nu(x)$ is defined by taking
$$k_1=-(k+2)\delta,~~~k_2=-(k+1)\delta,~~~k_3=(k+1)\delta,~~~k_4=(k+2)\delta$$
and the function $\mu(y)$ is defined by the same way as shown in {\it Step 1.}
Then we take
$\lambda^+_{i}=-k\delta+(i-1)\delta$ for all $i=1,...,2k+1$
to desingularize the tangent point $O$ and then for each $\ell\in\{1,...,k+1\}$,
we can define the function $\psi^+(x,{\boldsymbol l}^+)$ such that
there is grazing loop connecting $\ell$ tangent points. Further, we can
continue to perturb some $l^+_{i,3}$ such that the grazing loop becomes
a sliding loop connecting $\ell$ tangent points.
Note that there is no standard limit cycle in $\Sigma^+$ because $\sigma<0$.
Thus, the cases that $\sigma<0$ are proved.

When $\sigma>0$, it can be proved by a similar way that there exist appropriate
${\boldsymbol \lambda}^+$ and functions $\psi^+(x,{\boldsymbol l}^+)$, $\nu(x)$,
$\mu(y)$ such that there is a sliding loop connecting several tangent points
and a standard limit cycle in $\Sigma^+$ because of the unstability of $L_I$.
Thus, the grazing loop $L_{I}$ satisfying $g^-(0,0)>0$ is analyzed for 
all $k\in{\mathbb Z}^+$. For the cases that $g^-(0,0)<0$, analysis is similar
and we omit the statements. 
\end{proof}

For the system~\eqref{pws1} satisfying assumptions {\bf A}$_1$ and {\bf A}$_3$,
it takes the following form
\begin{equation}
\begin{aligned}
            \left(\begin{array}{c}
            \dot{x}\\
            \dot{y}\\
        \end{array}\right)=\left\{
        \begin{aligned}
    &       \left(\begin{array}{c}
                f^+(x,y)\\
                \phi^+(x,y)x^{2k+1}+\Upsilon^+(x,y)
            \end{array}\right)   &&\mathrm{if~}(x,y)\in\Sigma^+,\\
    &       \left(\begin{array}{c}
                f^-(x,y)\\
                \phi^-(x,y)x+\Upsilon^-(x,y)
            \end{array}\right)   &&\mathrm{if~}(x,y)\in\Sigma^-,
        \end{aligned} \right.
    \end{aligned}
\label{pws-thm2}
\end{equation}
where $k\in{\mathbb Z}^+$, $\phi^\pm(0,0)\ne0$ and $\Upsilon^\pm(x,0)\equiv0$.
The proof of Theorem~\ref{thm2} is given by analyzing the following
perturbation system
\begin{equation}
\begin{aligned}
            \left(\begin{array}{c}
            \dot{x}\\
            \dot{y}\\
        \end{array}\right)\!=\!\left\{
        \begin{aligned}
    &       \left(\begin{array}{c}
                f^+(x,y)\\
                \phi^+(x,y)x^{2k+1}+\Upsilon^+(x,y)+{\cal P}^+(x,y)
            \end{array}\right)   &&\mathrm{if}~(x,y)\in\Sigma^+,\\
    &       \left(\begin{array}{c}
                f^-(x,y)\\
                \phi^-(x,y)x+\Upsilon^-(x,y)+{\cal P}^-(x,y)
            \end{array}\right)   &&\mathrm{if}~(x,y)\in\Sigma^-,
        \end{aligned} \right.
    \end{aligned}
\label{pws-thm2-pf}
\end{equation}
where ${\cal P}^+(x,y), \nu(x), \mu(y)$ are defined in \eqref{perturb-fun}, \eqref{df-nu}, \eqref{df-mu} respectively and
${\cal P}^-(x,y):=\phi^-(x,y)(x-\lambda^-_1)-\phi^-(x,y)x+\psi^-(x,{\boldsymbol l}^-)$
for $\lambda^-_1\in{\mathbb R}$, $\psi^\pm(x,{\boldsymbol l}^\pm)$ are
of form \eqref{df-psi}. 

\begin{proof}[Proof of Theorem \ref{thm2}]
For the conclusion (a), we first consider the cases that $\sigma<0$ and 
begin with $k=1$. Now $\ell\in\{1,2\}$ and the proof proceeds the 
following four steps.

{\it Step 1. Determine appropriate functions $\nu(x)$ and $\mu(y)$.}
As indicated in the {\it Step 1.} of the proof of Theorem~\ref{thm1},
the function $\nu(x)$ is defined by taking
$k_1=-3\delta, k_2=-2\delta, k_3=2\delta, k_4=3\delta$
and the function $\mu(y)$ is defined by taking
$r_1=\delta, r_2=\delta+c$
for a sufficiently small $\delta>0$,
where $c$ is a positive constant such that the region
$\{(x,y):~x\in(-3\delta,3\delta),y\in(\delta,\delta+c)\}$ is in the
interior of $L_{II}$. Corresponding properties of system~\eqref{pws-thm2-pf}
are shown in the proof of Theorem~\ref{thm1} and we omit these for simplification.

{\it Step 2. Determine appropriate parameters ${\boldsymbol \lambda}^+$
and $\lambda^-_1$ in system~\eqref{pws-thm2-pf}.} We first take
$\psi^\pm(x,{\boldsymbol l}^\pm)\equiv0$ in system~\eqref{pws-thm2-pf},
i.e., obtain the following transitional system between \eqref{pws-thm2}
and \eqref{pws-thm2-pf} in the region $\{(x,y):~x\in(-2\delta,2\delta),y\in(-\infty,\delta)\}$
\begin{equation}
\begin{aligned}
            \left(\begin{array}{c}
            \dot{x}\\
            \dot{y}\\
        \end{array}\right)=\left\{
        \begin{aligned}
    &       \left(\begin{array}{c}
                f^+(x,y)\\
                \phi^+(x,y)\!\!\prod\limits_{i=1}^{2k+1}\!\!(x-\lambda^+_i)+\Upsilon^+(x,y)
            \end{array}\right)   &&\mathrm{if}~(x,y)\in\Sigma^+,\\
    &       \left(\begin{array}{c}
                f^-(x,y)\\
                \phi^-(x,y)(x-\lambda^-_1)+\Upsilon^-(x,y)
            \end{array}\right)   &&\mathrm{if}~(x,y)\in\Sigma^-.
        \end{aligned} \right.
    \end{aligned}
\label{tpws-thm2-pf}
\end{equation}
Then we take $\lambda^-_1=0$ and consequently, there is an invisible tangent point
$(0,0)$ for the lower subsystem of \eqref{tpws-thm2-pf}. As indicated in
\cite{Filippov88}, for any small $x>0$, the orbit $\widehat\gamma^-(t,x,0)$
intersects $\Sigma$ at some point $(\widehat x,0)$ satisfying
$
\widehat x=-x+O(x^2)
$
for forward direction, where
$$\widehat\gamma^-(t,x_0,y_0):=(\widehat\gamma^-_1(t,x_0,y_0),\widehat\gamma^-_2(t,x_0,y_0))^\top$$
denotes the orbit of the lower subsystem of \eqref{tpws-thm2-pf} with initial
value $(x_0,y_0)$. Further, for a sufficiently small $\delta>0$,
denote the intersection between the orbit $\widehat \gamma^+(t,\delta,0)$
and $\Sigma$ as $(\widehat \delta,0)=(-\delta+O(\delta^2),0)$ for forward direction. Further,
we take
$
\lambda^+_1=\widehat \delta, \lambda^+_2=0, \lambda^+_3=\delta
$
in system~\eqref{tpws-thm2-pf}, which implies that there are two
visible tangent points $(\lambda^+_1,0)$, $(\lambda^+_3,0)$ and
an invisible tangent point $(\lambda^+_2,0)$. Let
$\widehat\gamma^+(t,x_0,y_0)$ denote the orbit of upper subsystem
of transitional system~\eqref{tpws-thm2-pf} with initial value $(x_0,y_0)$.
Then we take three small vertical line segments
$S(c):=\{(c,y):~y\in(0,2\delta^3)\}$, where $c=-3\delta,0,3\delta$,
and check that the orbit $\widehat\gamma^+(t,-3\delta,\delta^3)$ intersects
$S(3\delta)$ and $S(0)$ at some point $(3\delta,\widehat{\cal Y}_1)$
and $(0,\widehat{\cal Y}_2)$ for forward direction respectively. Denote the
intersection between the orbit $\widehat\gamma^+(t,-3\delta,\delta^3)$ and
$S(0)$ as $(0,\widehat{\cal Y}_3)$ for backward direction. It can be proved
by a similar way in Theorem~\ref{thm1} that $
\widehat{\cal Y}_i=\exp\{\sigma\}\delta^3+O(\delta^4)>0
$ for $i=1,2,3$.

{\it Step 3. Determine functions $\psi^\pm(x,{\boldsymbol l}^\pm)$
in system~\eqref{pws-thm2-pf}.} The function $\psi^+(x,{\boldsymbol l}^+)$ is defined
by taking $s^+=4$,
\begin{equation*}
\begin{aligned}
&l^+_{1,1}=-2\delta,~~~l^+_{1,2}=\widehat\delta,&&l^+_{2,1}=\widehat\delta,~~~l^+_{2,2}=0,\\
&l^+_{3,1}=0,~~~~~~~l^+_{3,2}=\delta,&&l^+_{4,1}=\delta,~~~l^+_{4,2}=2\delta
\end{aligned}
\end{equation*}
and $l^+_{i,3}$ $(i=1,...,4)$ are undetermined. As indicated in the proof of
Theorem~\ref{thm1},
$\widetilde\gamma^+(t,-3\delta,\delta^3)\equiv\widehat\gamma^+(t,-3\delta,\delta^3)$
outside the region $\{(x,y):~x\in(-3\delta,3\delta),y\in(0,\delta+c)\}$,
where $\widetilde\gamma^+(t,x_0,y_0)$ denotes the orbit of the upper
subsystem of \eqref{pws-thm2-pf} with initial value $(x_0,y_0)$.
Further, it can be proved by lemma \ref{lm-TP} that there exist some
$l^+_{i,3}$ ($i=1,...,4$) such that the orbit $\widetilde\gamma^+(t,-3\delta,\delta^3)$
passes through the visible tangent point $(\lambda^+_1,0)$ and intersects
$S(0)$ at some point $(0,\widetilde{\cal Y}_2)$ satisfying
$\widetilde{\cal Y}_2<\min\{\widehat{\cal Y}_2,\widehat{\cal Y}_3\}$
for backward direction. And simultaneously, the orbit $\widetilde\gamma^+(t,-3\delta,\delta^3)$
passes through the visible tangent point $(\lambda^+_3, 0)$ and intersects
$S(0)$ at some point $(0,\widetilde{\cal Y}_1)$ satisfying
$\widetilde{\cal Y}_1>\max\{\widehat{\cal Y}_2,\widehat{\cal Y}_3\}$
for forward direction. Thus, we obtain 
\begin{equation}
\widetilde {\cal Y}_2<\widetilde {\cal Y}_1
\label{L-outter}
\end{equation}
and there is a critical loop connecting two tangent point of multiplicity $(1,0)$
bifurcating from the grazing loop $L_{II}$ by taking
$\psi^-(x,{\boldsymbol l}^-)\equiv0$ as shown in Figure~\ref{Fig-THM2-1}
(denoted as $L^{cri}_1$). In fact, there is a standard limit cycle
because of the stability of $L_{II}$ and \eqref{L-outter}.
\begin{figure}[htp]
\centering
\includegraphics[scale=0.45]{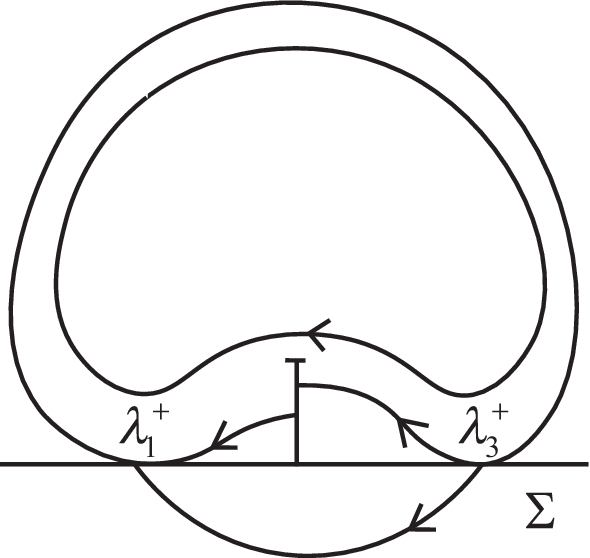}
\caption{The grazing loop $L^{cri}_1$ and the standard limit cycle}
\label{Fig-THM2-1}
\end{figure}

{\it Step 4. Perturb $L^{cri}_1$ to obtain crossing limit cycles and
sliding loops.} In order to characterize $L^{cri}_1$, we take a small vertical line segment
$S(a):=\{(a,y):~y\in(b-\epsilon,b+\epsilon)\}$
at some point $(a,b)$ on $L^{cri}_1$ satisfying $f^+(a,b)>0$.
Further for $(x,0)\in S_1:=\{(x,0):~x\in(\lambda^+_1-\epsilon,\lambda^+_1)\}$,
the orbit $\widetilde\gamma^+(t,x,0)$ intersects $S(a)$ and
$S_2:=\{(x,0):~x\in(\lambda^+_3,\lambda^+_3+\epsilon)\}$ for forward direction.
This implies that we can define two half return maps
\begin{equation*}
\begin{aligned}
P^+_1(x):~~&S_1&&\to ~~~S(a)\\
           &x&&\mapsto ~~~b+V^+_1(x)
\end{aligned}
\end{equation*}
and
\begin{equation*}
\begin{aligned}
P^+_2(x):~~&S_2&&\to ~~~S(a)\\
           &x&&\mapsto ~~~b+V^+_2(x)
\end{aligned}
\end{equation*}
where $V^+_1(x)$ (resp. $V^+_2(x)$) is the transition map from $S_1$ (resp. $S_2$)
to $S(a)$. Let $\widetilde{\cal Z}^\pm(x,y):=(\widetilde{\cal Z}^\pm_1(x,y),\widetilde{\cal Z}^\pm_2(x,y))^\top$
denote the vector fields of the upper and lower subsystem of \eqref{pws-thm2-pf}. Further,
it can be proved by lemma~\ref{lm-TM} that
\begin{equation*}
\begin{aligned}
&P^+_1(x)=b+V^+_{1,2}(x-\lambda^+_1)^2+O((x-\lambda^+_1)^3),\\
&P^+_2(x)=b+V^+_{2,2}(x-\lambda^+_3)^2+O((x-\lambda^+_3)^3),
\end{aligned}
\end{equation*}
where
\begin{equation*}
\begin{aligned}
&V^+_{1,2}:=\frac{\widetilde{\cal Z}^+_{2x}(\lambda^+_1,0)}{2\widetilde{\cal Z}^+_1(a,b)}\exp\left\{\int_0^{{\cal T}_a}\frac{\partial \widetilde {\cal Z}^+_1(\widetilde\gamma^+(s,\lambda^+_1,0))}{\partial x}+\frac{\partial \widetilde {\cal Z}^+_2(\widetilde\gamma^+(s,\lambda^+_1,0))}{\partial y}ds\right\},\\
&V^+_{2,2}:=\frac{\widetilde{\cal Z}^+_{2x}(\lambda^+_3,0)}{2\widetilde{\cal Z}^+_1(a,b)}\exp\left\{\int_{{\cal T}}^{{\cal T}_a}\frac{\partial \widetilde {\cal Z}^+_1(\widetilde\gamma^+(s,\lambda^+_1,0))}{\partial x}+\frac{\partial \widetilde {\cal Z}^+_2(\widetilde\gamma^+(s,\lambda^+_1,0))}{\partial y}ds\right\}
\end{aligned}
\end{equation*}
and ${\cal T}_a$ (resp. ${\cal T}$) is the time when the orbit
$\widetilde\gamma^+(t,\lambda^+_1,0)$ intersects $S(a)$ (resp. $S_2$).

For the lower subsystem of \eqref{pws-thm2-pf}, Since there is an invisible tangent
point $(0,0)$ and $S_i$ ($i=1,2$) are near it, we define the half return map
\begin{equation*}
\begin{aligned}
P^-(x):~~&S_1&&\to ~~~S_2\\
           &x&&\mapsto ~~~-x+O(x^2)
\end{aligned}
\end{equation*}
as indicated in \cite{Filippov88}. Further, the displacement function is defined
as follows
\begin{equation*}
\begin{aligned}
D(x)
& := P^+_1(x)-P^+_2\circ P^-(x) \\
& = (V^+_{1,2}-V^+_{2,2})(x-\lambda^+_1)^2+O((x-\lambda^+_1)^3)\\
& = \frac{1}{2\widetilde{\cal Z}^+_1(a,b)}\frac{\widetilde{\cal Z}^+_{2x}(\lambda^+_1,0)}{E({\cal T},{\cal T}_a))}\left(\frac{E(0,{\cal T})}{E({\cal T},{\cal T}_a)}-\frac{\widetilde{\cal Z}^+_{2x}(\lambda^+_3,0)}{\widetilde{\cal Z}^+_{2x}(\lambda^+_1,0)}\right)(x-\lambda^+_1)^2+O((x-\lambda^+_1)^3),
\end{aligned}
\end{equation*}
where
$$E(x_0,x_1):=\exp\left\{\int_{x_0}^{x_1}\frac{\partial \widetilde {\cal Z}^+_1(\widetilde\gamma^+(s,\lambda^+_1,0))}{\partial x}+\frac{\partial \widetilde {\cal Z}^+_2(\widetilde\gamma^+(s,\lambda^+_1,0))}{\partial y}ds\right\}.$$
It is not difficult to check that
\begin{equation*}
\frac{E(0,{\cal T})}{E^({\cal T},{\cal T}_a)} = \exp\left\{\int_{0}^{{\cal T}}\frac{\partial \widetilde {\cal Z}^+_1(\widetilde\gamma^+(s,\lambda^+_1,0))}{\partial x}+\frac{\partial \widetilde {\cal Z}^+_2(\widetilde\gamma^+(s,\lambda^+_1,0))}{\partial y}ds\right\}
\end{equation*}
and it is near the constant $\exp\{\sigma\}$ for sufficiently small $\delta$,
where $\sigma<0$, i.e., $\exp\{\sigma\}<1$.
On the other hand, it can be obtained by a straight computation that
\begin{equation*}
\begin{aligned}
\frac{\widetilde{\cal Z}^+_{2x}(\lambda^+_3,0)}{\widetilde{\cal Z}^+_{2x}(\lambda^+_1,0)}
& = -\frac{\phi^+(\lambda^+_3,0)(\lambda^+_3-\lambda^+_2)}{\phi^+(\lambda^+_1,0)(\lambda^+_1-\lambda^+_2)}\\
& = \frac{\phi^+(\delta,0)}{\phi^+(-\delta+O(\delta^2),0)(1+O(\delta^2))}\to 1
\end{aligned}
\end{equation*}
as $\delta\to0$. Thus for sufficiently small $\delta>0$,
the displacement function $D(x)$ starts at degree $2$ because
$V^+_{1,2}-V^+_{2,2}>0.$
This implies that $L^{cri}_1$ is unstable, i.e., there is a point
$(Q_1,0)\in S_2$ such that the orbit starting from it return to
$S_2$ at some point $(\widetilde Q_1,0)$ satisfying
\begin{equation}
Q_1<\widetilde Q_1
\label{outter}
\end{equation}
for forward direction. And simultaneously, the sum of the numbers
of crossing limit cycles and sliding loops is at most two under perturbations.

Then we prove the reachability by perturbing some $l^+_{i,3}$ and defining
appropriate nonzero function $\psi^-(x,{\boldsymbol l}^-)$.
We define $\psi^-(x,{\boldsymbol l}^-)$ by taking
$s^-=1$ and
$
l^-_{1,1}=0, l^-_{1,2}=\delta, l^-_{1,3}=\epsilon_1,
$
where $\epsilon_1>0$ and is sufficiently small.
It can be proved by the analysis in Lemma \ref{lm-TP} that
the orbit $\widetilde\gamma^-(t,\widehat\delta,0)$ intersects $\Sigma$
at some point $(\widetilde\delta,0)$ satisfying $0<\widetilde\delta-\lambda^+_3\ll1$
for backward direction, where $\widetilde\gamma^-(t,x_0,y_0)$
denotes the orbit of the lower subsystem of \eqref{pws-thm2-pf} with initial value
$(x_0,y_0)$. Further, we can perturb $l^+_{4,3}$ and keep $l^+_{i,3}$ ($i=1,2,3$)
such that the orbit $\widetilde\gamma^+(t,-3\delta,\delta^3)$ passes through
the tangent point $(\lambda^+_1,0)$ for backward direction and exactly intersects
$\Sigma$ at $(\widetilde\delta,0)$ for forward direction, i.e.,
there is a critical loop connecting the tangent point $(\lambda^+_1,0)$
(denoted as $L^{cri}(\lambda^+_1)$). As indicated in \cite{Ponce15,Han13},
the loop $L^{cri}(\lambda^+_1)$ is stable, i.e. there is a point $(P_1,0)\in S_2$
such that the orbit starting from it returns to $S_2$ at some point $(\widetilde P_1,0)$
satisfying
\begin{equation}
P_1>\widetilde P_1
\label{inner}
\end{equation}
for forward direction. Further, we continue to perturb the loop $L^{cri}(\lambda^+_1)$
by perturbing $l^+_{4,3}$. For the case that $\ell=1$, we can perturb
$l^+_{4,3}$ such that the orbit $\widetilde\gamma^+(t,-3\delta,\delta^3)$
intersects $\Sigma$ at some point $(\widetilde\delta^-,0)$ satisfying
$
\lambda^+_3<\widetilde\delta^-<\widetilde\delta
$
for forward direction. Further, this orbit enters the region $\Sigma^-$ and then
intersects $\Sigma$ at some point $(\widetilde\lambda^+_1,0)$ satisfying
$0<\widetilde\lambda^+_1-\lambda^+_1\ll1$. It is not hard to check that
$(\widetilde\lambda^+_1,0)\in\Sigma_s$ and the first component of the sliding
vector field ${\cal X}_s(\widetilde\lambda^+_1)<0$. Thus, this orbit returns to
the tangent point $(\lambda^+_1,0)$, i.e., there is a sliding loop. Meanwhile,
since perturbations of $l^+_{4,3}$ and $l^-_{1,3}$ are sufficiently small,
the inequalities \eqref{outter} and \eqref{inner} still hold and consequently,
there is a crossing limit cycle which intersects $\Sigma$ at some point $({\cal P},0)$
satisfying $P_1<{\cal P}<Q_1$. There is a standard limit cycle in $\Sigma^+$
because of the hyperbolicity of $L_{II}$. The case that $\ell=1$ is proved and
an example is given in Figure~\ref{Fig-THM2-2}(a). For the case that $\ell=2$,
we can perturb $l^+_{4,3}$ such that the orbit $\widetilde\gamma^+(t,-3\delta,\delta^3)$
intersects $\Sigma$ at some point $(\widetilde\delta^+,0)$ satisfying
$0<\widetilde\delta^+-\widetilde\delta\ll1,$
i.e., there is a crossing limit cycle bifurcating from $L^{cri}(\lambda^+_1)$.
And simultaneously, inequalities \eqref{L-outter}, \eqref{outter} and \eqref{inner}
hold because of the small perturbations, and further, there are two crossing limit
cycles and a standard limit cycle. The case that $\ell=2$ is proved and an example
is given in Figure~\ref{Fig-THM2-2}(b).
\begin{figure}[h]
\centering
\subfigure[$\beta_c\ge1$ and $\beta_s=1$]
 {
  \scalebox{0.44}[0.44]{
   \includegraphics{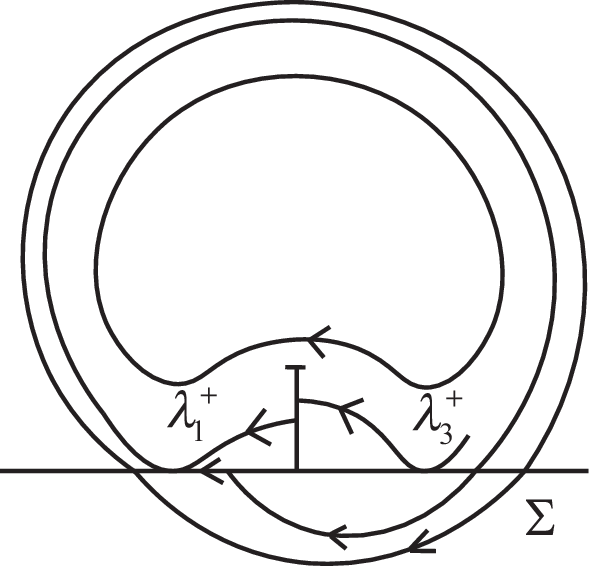}}}~~~~~~~~~
\subfigure[$\beta_c\ge2$ and $\beta_s=0$]
 {
  \scalebox{0.44}[0.44]{
   \includegraphics{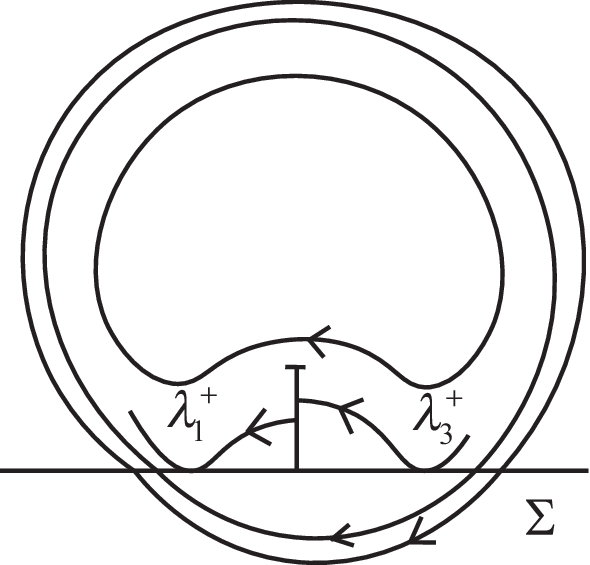}}}
   \caption{The standard limit cycle, crossing limit cycles and sliding loops of \eqref{pws-thm2-pf}}
\label{Fig-THM2-2}
\end{figure}

Then we consider the case that $k=2$ and in this case, $\ell\in\{2,3,4\}$.
The proof is similar with ones of case that $k=1$ and also proceeds four steps.

{\it Step 1. Determine appropriate functions $\nu(x)$ and $\mu(y)$.}
The function $\nu(x)$ is defined by taking
$k_1=-4\delta, k_2=-3\delta, k_3=3\delta, k_4=4\delta$
and the function $\mu(y)$ is defined by taking
$r_1=\delta, r_2=\delta+c$
for a sufficiently small $\delta>0$,
where $c$ is a positive constant such that the region
$\{(x,y):~x\in(-3\delta,3\delta),y\in(\delta,\delta+c)\}$ is in the
interior of $L_{II}$.

{\it Step 2. Determine appropriate parameter ${\boldsymbol \lambda}^+$ and
$\lambda^-_1$ in system \eqref{pws-thm2-pf}.} Similarly, we consider the
transitional system~\eqref{tpws-thm2-pf} and take $\lambda^-=0$. Further,
for sufficiently small $\delta>0$, denote the intersection between
the orbit $\widehat\gamma^-(t,2\delta,0)$ and $\Sigma$ as $(\widehat\delta,0)$.
Further, we take
$
\lambda^+_1=\widehat\delta, \lambda^+_2=-\delta, \lambda^+_3=0, \lambda^+_4=\delta, \lambda^+_5=2\delta
$
in system~\eqref{tpws-thm2-pf}. Thus, there are three visible tangent points
$(\lambda^+_1,0)$, $(\lambda^+_3,0)$ and $(\lambda^+_5,0)$ and two invisible
tangent points $(\lambda^+_2,0)$ and $(\lambda^+_4,0)$. Further, we take
four small vertical line segments
$
S(c):=\{(c,y):~y\in(0,3\delta^5)\}$, where $c=-3\delta, -\delta, \delta, 3\delta$.

{\it Step 3. Determine appropriate functions $\psi^\pm(x,{\boldsymbol l}^\pm)$
in system~\eqref{pws-thm2-pf}.} The function $\psi^+(x,{\boldsymbol l}^+)$ by taking $s^+=6$ and
\begin{equation*}
\begin{aligned}
&l^+_{1,1}=-3\delta,~~~&&l^+_{1,2}=\widehat\delta,~~~
&&l^+_{2,1}=\widehat\delta,~~~&&l^+_{2,2}=-\delta,~~~
&&l^+_{3,1}=-\delta,~~~~&&l^+_{3,2}=0,\\
&l^+_{4,1}=0,~~~&&l^+_{4,2}=\delta,~~~
&&l^+_{5,1}=\delta,~~~&&l^+_{5,2}=2\delta,~~~
&&l^+_{6,1}=2\delta,~~~&&l^+_{6,2}=3\delta
\end{aligned}
\end{equation*}
and $l^+_{i,3}$ ($i=1,...,6$) are determined by the following analysis.
It can be proved by Lemma \ref{lm-TP} that there exist appropriate
$l^+_{1,3}$, $l^+_{2,3}$, $l^+_{5,3}$ and $l^+_{6,3}$ such that the orbit
$\widetilde\gamma^+(t,-4\delta,\delta^5)$ passes through the tangent point
$(\lambda^+_5,0)$ and intersects $S(\delta)$ at some point
$(\delta,\widetilde{\cal Y}_{1,1})$ satisfying $\widetilde{\cal Y}_{1,1}>0$
for forward direction. And simultaneously, this orbit passes through the
tangent point $(\lambda^+_1,0)$ and intersects $S(-\delta)$ at some point
$(-\delta,\widetilde{\cal Y}_{1,2})$ satisfying $\widetilde{\cal Y}_{1,2}>0$
for backward direction. Further, there is a critical loop connecting
two tangent points of multiplicity $(1,0)$ (denoted as $L^{cri}_1$)
as shown in the Figure~\ref{Fig-THM2-3}(a).
\begin{figure}[h]
\centering
\subfigure[$L^{cri}_1$]
 {
  \scalebox{0.44}[0.44]{
   \includegraphics{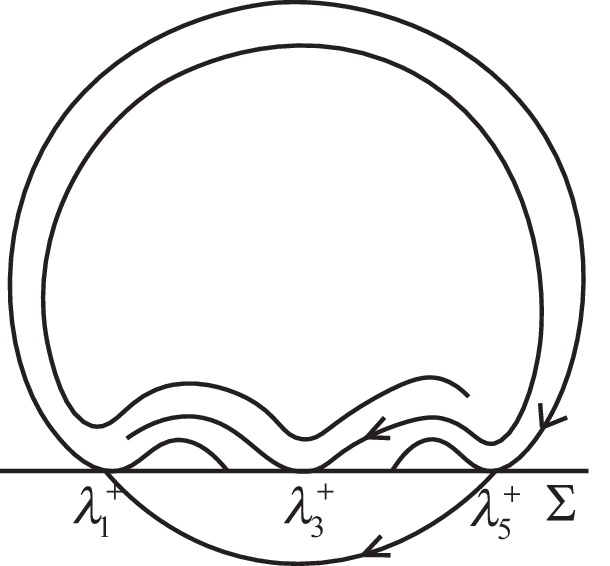}}}~~~~~~~~~
\subfigure[$L^{cri}_1$ and $L^{gra}_*$]
 {
  \scalebox{0.44}[0.44]{
   \includegraphics{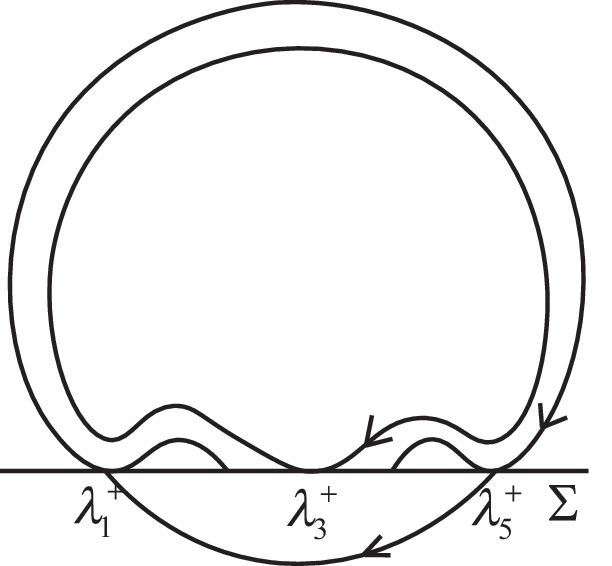}}}
   \caption{The critical loops and grazing loops of \eqref{pws-thm2-pf}}
\label{Fig-THM2-3}
\end{figure}

Then for the orbit $\widetilde\gamma^+(t,-4\delta,2\delta^5)$ and the
parameters $l^+_{i,4}$ ($i=1,2,5,6$), it is not hard to check that
this orbit intersects $S(\delta)$ at some point
$(\delta,\widetilde{\cal Y}_{2,1})$ satisfying
$0<\widetilde{\cal Y}_{1,1}<\widetilde{\cal Y}_{2,1}$ for forward
direction and simultaneously, intersects $S(-\delta)$ at some point
$(-\delta,\widetilde{\cal Y}_{2,2})$ satisfying
$0<\widetilde{\cal Y}_{1,2}<\widetilde{\cal Y}_{2,2}$ for backward direction.
Similarly, it can be proved by lemma~\ref{lm-TP} that there exist $l^+_{3,3}$
and $l^+_{4,3}$ such that the orbit $\widetilde\gamma^+(t,-4\delta,2\delta^5)$
passes through the tangent point $(\lambda^+_3,0)=(0,0)$ for both directions,
i.e., there is a grazing loop connecting a tangent point of multiplicity
$(1,1)$ (denoted as $L^{gra}_*$) as shown in the Figure~\ref{Fig-THM2-3}(b).

{\it Step 4. Perturb $L^{cri}_1$ and $L^{gra}_*$ to obtain crossing limit cycles
and sliding loops.} As indicated in the proof above, we can define the
displacement function with respect to the loop $L^{cri}_1$ and find that
the Taylor expand starts degree $2$ because $\sigma<0$. Meanwhile,
as indicated in \cite{Li20}, we also can define the displacement function
with respect to $L^{gra}_*$ and find that the Taylor expand starts degree $2$
because $\sigma<0$.

Further, we define the nonzero function $\psi^-(x,{\boldsymbol l}^-)$ and
perturb $l^+_{6,3}$ to perturb $L^{cri}_1$. Precisely, the function
$\psi^-(x,{\boldsymbol l}^-)$ is defined by taking $s^-=1$,
$
l^-_{1,1}=\delta, l^-_{1,2}=2\delta
$
and sufficiently small $l^-_{1,3}\ne0$. Then it can be proved by the
analysis in {\it Step 3} of case $k=1$ that for each $\ell\in\{1,2\}$,
there exist some $l^+_{6,3}$ and $l^-_{1,3}$ such that there are
$\ell$ crossing limit cycles and $2-\ell$ sliding loops bifurcating from
$L^{cri}_1$. Note that after perturbing $l^+_{6,3}$,
the loop $L^{gra}_*$ may break. But for the perturbed $l^+_{6,3}$,
we can take a new $l^+_{4,3}$ to keep $L^{gra}_*$ by the analysis in
Lemma \ref{lm-TP}. Further, as indicated in \cite{Li20},
for each $\ell\in\{1,2\}$ there exists some $l^+_{4,4}$ and
$\lambda^-_1\ne0$ such that there are $\ell$ crossing limit cycles
and $2-\ell$ sliding loops bifurcating from $L^{gra}_*$.
Since the perturbations of $l^+_{4,4}$ and $\lambda^-_1$ are sufficiently
small, the crossing limit cycles and sliding loops bifurcating from
$L^{cri}_1$ are preserved. Thus, the cases that $\ell=2,3,4$ are proved because
the bifurcations of $L^{cri}_1$ and $L^{gra}_*$ are independent of each other.
The examples are shown in Figure~\ref{Fig-THM2-4}.
\begin{figure}[h]
\centering
\subfigure[$\beta_c\ge2$ and $\beta_s=2$]
 {
  \scalebox{0.45}[0.45]{
   \includegraphics{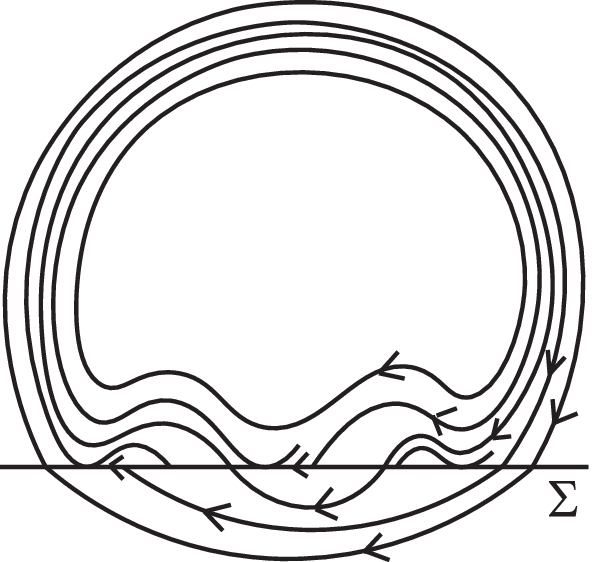}}}
\subfigure[$\beta_c\ge3$ and $\beta_s=1$]
 {
  \scalebox{0.45}[0.45]{
   \includegraphics{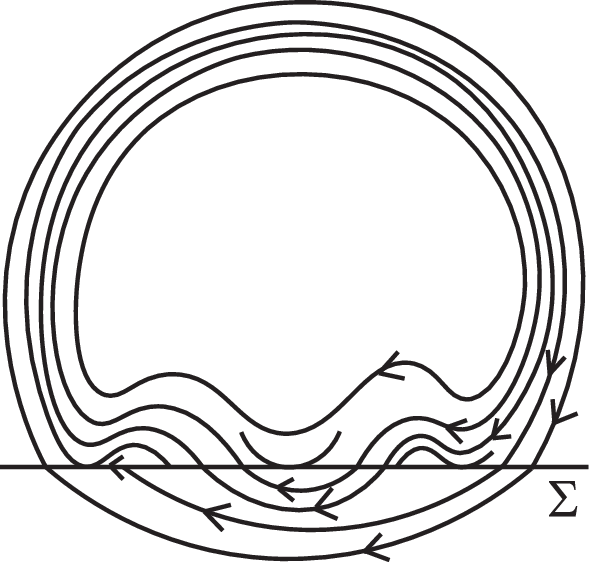}}}
\subfigure[$\beta_c\ge4$ and $\beta_s=0$]
{
  \scalebox{0.45}[0.45]{
   \includegraphics{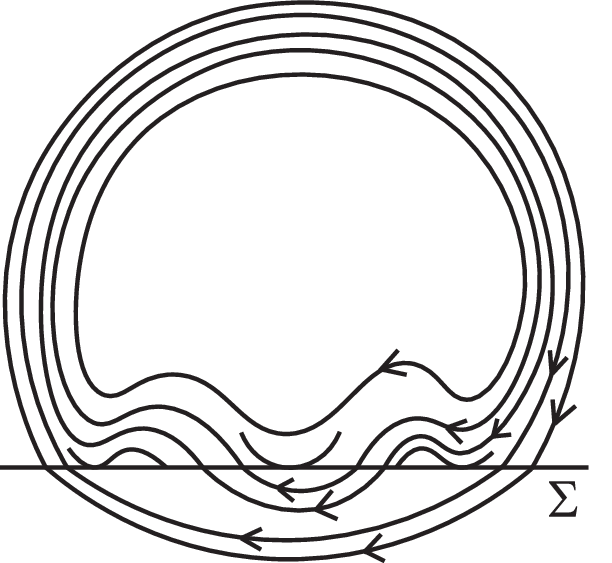}}}
   \caption{The standard limit cycle, crossing limit cycles and sliding loops of \eqref{pws-thm2-pf}}
\label{Fig-THM2-4}
\end{figure}

For the cases that $k$ is odd, it can be proved by the analysis above
that there are exactly $1+\lfloor k/2\rfloor$ critical loops connecting
two tangent points of multiplicity $(1,0)$ (denoted as $L^{cri}_i$
($i=1,...,1+\lfloor k/2\rfloor$)). Meanwhile, locations of them can be
written as
\begin{equation*}
L^{cri}_{1+\lfloor k/2\rfloor}~\hookrightarrow~L^{cri}_{\lfloor k/2\rfloor}~
\hookrightarrow~...~\hookrightarrow~L^{cri}_{2}~
\hookrightarrow~L^{cri}_{1},
\end{equation*}
where $L^{cri}_{i+1}~\hookrightarrow~L^{cri}_{i}$ denotes
the loop $L^{cri}_{i+1}$ is in the region surrounded by $L^{cri}_{i}$.
Thus, the result \eqref{Res} can be obtained by perturbing them sequentially
from $L^{cri}_1$ to $L^{cri}_{1+\lfloor k/2\rfloor}$.

For the case that $k$ is even, it can be proved by the analysis above
that there are exactly $k/2$ critical loops connecting two tangent points
of multiplicity $(1,0)$ (denoted as $L^{cri}_i$ ($i=1,...,k/2$)) and
a grazing loop connecting a tangent point of multiplicity $(1,1)$
(denoted as $L^{gra}_*$). Meanwhile, locations of them can be written as
\begin{equation*}
L^{gra}_*~\hookrightarrow~L^{cri}_{k/2}~\hookrightarrow~...
~\hookrightarrow~L^{cri}_{2}~\hookrightarrow~L^{cri}_{1}.
\end{equation*}
Thus, the result \eqref{Res} can be obtained by perturbing them
sequentially from $L^{cri}_1$ to $L^{cri}_{k/2}$ and to $L^{gra}_*$.
The cases that $\sigma>0$ can be proved by a similar way and we omit the
statements. Thus conclusion (a) is proved.

For the conclusion (b), if $\sigma<0$ (resp. $\sigma>0$), we can first 
take $\psi^-(x,{\boldsymbol l}^-)\equiv0$ and $\lambda^-_1<0$ (resp. $>0$) 
in the lower subsystem of \eqref{pws-thm2-pf} such that the 
additional condition in Theorem~\ref{thm1} holds, i.e., 
$-\phi^-(0,0)\lambda^-_1$ and $\sigma$ have the same sign. 
Thus, there is a grazing loop connecting a tangent point of multiplicity 
$(2k+1,0)$ in system~\eqref{pws-thm2-pf} with ${\cal P}^+(x,y)\equiv0$.
Then the conclusion (b) can be proved by a similar way in
the proof of Theorem~\ref{thm1}.
\end{proof}

\section{Conclusions and remarks}
\setcounter{equation}{0}
\setcounter{lm}{0}
\setcounter{thm}{0}
\setcounter{rmk}{0}
\setcounter{df}{0}
\setcounter{cor}{0}

In this paper, we focus on investigating question (Q)
and analyze bifurcations of grazing loops connecting a tangent point
$O:(0,0)$ of multiplicity $(2k+1,0)$ and $(2k+1,1)$ for general
$k\in{\mathbb Z}^+$. For the cases $(2k+1,0)$, the existence of the
sliding loop connecting several tangent points bifurcating from the grazing
loop $L_{I}$ is given in Theorem~\ref{thm1}, which generalizes the result
in the previous publications (e.g., \cite{Kuznetsov03,Teixeira11})
from $k=0$ to general $k\in{\mathbb Z}^+$. Meanwhile, the relationship
between the number of tangent points on the sliding loop and the
multiplicity $2k+1$ is also given. For the cases $(2k+1,1)$, we mainly
focus on the grazing loop $L_{II}$ in Figure~\ref{Fig-Gra}(c), i.e.,
the tangent point is invisible for the lower subsystem
and $f^-(0,0)<0$. Further, it is shown in Theorem~\ref{thm2} that
there are several crossing limit cycles, sliding loops and
one standard limit cycle bifurcating from the grazing loop $L_{II}$.
Meanwhile, the numbers of crossing limit cycles and sliding loops depend
on the multiplicity $2k+1$ as shown in \eqref{Res}. This generalizes the
result given in the previous publications (e.g., \cite{Han13,Li20}) from $k=0$ to
general $k\in{\mathbb Z}^+$. For the grazing loop $L_{II}$ in
Figure~\ref{Fig-Gra}(d), it is shown in Theorem~\ref{thm2} that there is one
standard limit cycle and a sliding loop connecting several tangent points
under some perturbation. This generalizes the result given in \cite{Li20} from
$k=0$ to general $k$.

As indicated in section~2, there are two main difficulties for question (Q)
arising from the general multiplicity $2k+1$ of the tangent point $O$,
i.e., the local and global recurrences are very complicated under perturbations.
The local recurrences are investigated in Lemma~\ref{lm-TP} by considering
a functional perturbation with functions, i.e., perturbation
system~\eqref{pws2-up-exunfold}. Further, it is shown in the proof of
Theorem~\ref{thm1} that the local recurrences is equivalent to the global
recurrences by introducing some localization functions, i.e., $\nu(x)$ and $\mu(y)$.
Precisely, we consider the perturbation systems~\eqref{pws-thm1-pf} with
\begin{equation*}
{\cal P}^+(x,y)=\left(\psi^+(x,{\boldsymbol l}^+)+\phi^+(x,y)\!\!\prod\limits_{i=1}^{2k+1}\!\!(x-\lambda^+_i)-\phi^+(x,y)x^{2k+1}\right)\nu(x)\mu(y),
\end{equation*}
where $\nu(x)$ and $\mu(y)$ are defined in \eqref{df-nu}
and \eqref{df-mu} respectively. Since the definitions of $\nu(x)$ and $\mu(y)$
lead to that $
{\cal P}^+(x,y)\equiv0
$ beside the region
$\{(x,y):~x\in(k_1,k_4),y\in(0,r_2)\}$
and $
{\cal P}^+(x,y)={\cal Q}^+(x,y)$ (see \eqref{fun-Q}) in the region $\{(x,y):~x\in(k_2,k_3),y\in(0,r_1)\}$,
we take some appropriate $k_i$ ($i=1,...,4$) in $\nu(x)$ and $r_1$, $r_2$
in $\mu(y)$ such that the perturbation function ${\cal P}^+(x,y)$ is
only used to control bifurcations of the tangent point $O$,
i.e., the functional perturbation is localized, as shown in the proof of
Theorem~\ref{thm1}. This ensures that the intersections
between the limit cycle $L_{I}$ and $\Sigma$ directly correspond to
interconnections between tangent orbits, i.e., an equivalent relationship is
established. For Theorem~\ref{thm2}, since the recurrences of the upper
and lower subsystem are clear, Poincar\'e return map can be defined correctly.
Further, we establish relationships between ${\cal P}^\pm(x,y)$
and the fixed points of Poincar\'e return map. Finally, Theorem \ref{thm2}
is proved by taking appropriate ${\cal P}^\pm(x,y)$. On the other
hand, we prove that the grazing loops shown in Figure~\ref{Fig-Gra}(c) 
indeed exhibits the mechanism by which crossing limit cycles and 
sliding limit cycles arise from the splitting of the tangent points. Meanwhile, 
we also present the perturbation system in which this mechanism occurs, i.e., 
the system~\eqref{pws-thm2-pf}, and in fact, this system is different from 
the perturbation system given in \cite{CFSCM}.


\end{document}